\documentclass[12pt,oneside]{amsart}
\usepackage{etnc-arxiv}

\def\marginpar#1{}
\renewcommand{\showlabelsetlabel}[1]{}

\begin{document}

\title{Tamagawa number formula for Jacobians}
\author{Tim Dokchitser}
\address{Department of Mathematics, University Walk, Bristol BS8 1TW, UK}
\email{tim.dokchitser@bristol.ac.uk}

\begin{abstract}
We give a product formula for the Tamagawa numbers of Jacobians over a discrete valuation
field with perfect residue field $k$. It comes as a product of four terms --- unipotent, toric, arithmetic
and (somewhat intricate) cohomological. It is proved by
(1) extending the classical flow-cut construction from semistable to arbitrary curves, which, by Raynaud's
results, gives the formula when $k$ is algebraically closed; (2) observing that this formula respects the natural
metric on the edges of the dual graph, which allows to quotient out the Galois action; (3) extending
Bosch--Liu's description of the cohomological term when $k$ is finite. In particular, this answers a question
of Bosch--Liu, and gives an alternative description of Poonen--Stoll's cohomological obstruction in
terms of characters.
\end{abstract}

\maketitle

\bigskip

\tableofcontents

\newpage

\section{Introduction}
\label{sintro}

Let $K$ be a discrete valuation field with ring of integers $\cO_K$ and perfect residue
field $k$.

Let $C/K$ be a (smooth projective geometrically connected) curve
with Jacobian $J/K$, and let $\J/\cO_K$ be the N\'eron
model of $J$. The \emph{component group} of $J$ is the quotient
$
  \phi_J \>=\> \J_k / \J_k^0
$
of the special fibre $\J_k$ by its identity component $\J_k^0$. It is a finite
\'etale group scheme over $k$, and we are interested in the number of its $k$-points:
$$
  \text{Tamagawa number} = |\phi_J(k)|.
$$
Following Raynaud \cite[9.6.1]{BLR} and Bosch--Liu \cite[\S1]{BL}, consider

\smallskip
\begin{tabular}{lll@{\qquad}l}
$\cC$    &=\!\!& a flat proper regular model of $C/K$ over $\cO_K$, with special fibre $\cC_k$,\\
$V$      &=\!\!& set of irreducible components $v$ of $\cC_{\bar k}$ (geometric components),\\
$v\cdot w$ &=\!\!& intersection pairing on $\cC_{\bar k}$ for $v,w\in V$ (self-intersection if $v=w$),\\
$\D$     &=\!\!& $(V,E)$ dual graph of the geometric special fibre with vertex set $V$, and\\
        &&  $v\cdot w$ edges between $v$ and $w$ for $v\ne w$ and no loops,\\
$\mathbf{m}=(m_v)\!\!$ &=\!& normalised multiplicities, so
  $\cC_{\bar k} \!=\! c\sum_{v\in V} m_v\, v$ with $c\!\ge\!1$ and $\gcd(\mathbf{m})\!=\!1$.\\
\end{tabular}

\smallskip
\noindent
The absolute Galois group $G=\Gal(\bar k/k)$ acts continuously on $V$, and,
blowing up intersection points if necessary, we assume that%
\footnote{(a) is automatic for a model with normal crossings; in fact, the action on $V$
  determines the component group; the action on $E$ only needs to exist. (b)
  is called acting without inversion, see Serre \cite[\S3.1]{Trees}.}
\begin{itemize}
\item[(a)] $v\cap w$ is reduced for all $v\ne w$
  (so the action of $G$ on $E$ is well-defined) and
\item[(b)] no element of $G$ swaps the two endpoints of any edge of $\D$.
\end{itemize}
Under these conditions we can define the quotient graph:
\smallskip

\begin{tabular}{lll@{\qquad}l}
$\DG$    &=& $(\VG,\EG)$ quotient graph, $\VG=V/G$, $\EG=E/G$ sets of orbits,\\
$r_v, r_e$ &=& orbit sizes for $v\in \VG, e\in \EG$,\\
$v_0$    &=& fixed choice of a representative in $V$ for every orbit $v\in \VG=V/G$,\\
$m_v$    &=& $m_{v_0}$ normalised multiplicity, constant on $G$-orbits,\\
$\deg v$ &=& degree of a vertex $v\in \VG$ in $\DG$, edges from $v$ to itself counted twice.\\
\end{tabular}

\begin{definition}[Length pairing]
Define a \emph{length pairing} on $\Z^{\EG}$ by
$$
  \lara: \Z^{\EG}\times \Z^{\EG} \to \Q, \qquad \langle e,e'\rangle =
    \bigleftchoice{1/(r_e m_v m_w)}{\text{if $e=e'$ is an edge $v$ to $w$ ($v,w\in \VG$),}}{0}{\text{if }e\ne e'.}
$$
Write $H_1(\DG,\Z)\subset C_1(\DG,\Z)=\Z^{\EG}$ for the first homology group of $\DG$ as a topological space.
\end{definition}

\begin{definition}[Cohomological term]
\label{defpsi}
A component orbit $v\in \VG$ is \emph{extremal} if $v_0\cdot v_0$ is odd,
$m_v$ is even and $r_v m_v/\gcd_{w\in \VG}(r_w m_w)$ is odd.
Let $\Gvee=\Hom(G,\Q/\Z)$ be the group of continuous characters of $G$, and write
$$
  \Psi_J=\Bigl\{\chi\in\Gvee \Bigm| \text{all }[\chi]_v \in \Z \text{ and }
    \sum_{v\in \VG\text{ extremal}}[\chi]_v \in 2\Z \Bigr\}
       \quad\text{where}\>\>\>\>
  [\chi]_v = \frac{m_v}{\order(\chi|_{\Stab_G(v_0)})}.
$$
\end{definition}

\namedthm{A}[Tamagawa number formula]
\label{etnc}
Let $K$ be a discrete valuation field with perfect residue field $k$, let
$C/K$ be a smooth projective geometrically connected curve with Jacobian~$J/K$,
and let $\cC/\cO_K$ be any flat proper regular model of $C$ satisfying
conditions~\textup{(a)} and~\textup{(b)}.
In the above notation, the Tamagawa number of $J/K$ is
$$
  |\phi_J(k)| \>=\>
    \underbrace{
      \prod_{v\in \VG} m_v^{\deg v - 2}
    \vphantom{\rule[-20pt]{0pt}{28pt}}
    }_{\substack{\text{\sf unipotent part} \\[2pt] \text{\sf $=\!1$ if $C/K$ semistable}}}
    \cdot
    \underbrace{
      \det\bigl(\langle\cdot,\cdot\rangle\big|_{H_1(\DG,\Z)}\bigr)
    \vphantom{\rule[-20pt]{0pt}{28pt}}
    }_{\substack{\text{\sf toric part} \\[2pt] \text{\sf $=\!1$ if $J/K$ has no toric part}}}
    \cdot
    \underbrace{
      \vphantom{\rule[-20pt]{0pt}{28pt}}
      \frac{\prod_{e\in \EG} r_e}{\prod_{v\in \VG} r_v}
    }_{\substack{\text{\sf arithmetic part} \\[2pt] \text{\sf $=\!1$ if $k=\bar{k}$}}}
    \cdot
    \underbrace{
      \gcd_{v\in \VG} (r_v m_v)
      \cdot
      |\Psi_J|\rlap{\ .}
    \vphantom{\rule[-20pt]{0pt}{28pt}}
    }_{\substack{\text{\sf cohomological part} \\[2pt] \text{\sf $=\!1$ if $C(K)\ne\emptyset$}}}
$$
In particular, the right-hand side is an integer, independent of the choice of a regular model.
\endnamedthm

\begin{example}[I$_n$]
Let $C/K$ be an elliptic curve with \emph{non-split} multiplicative
reduction. Its minimal and minimal normal crossings models coincide, and consist of an $n$-gon of $\P^1$s
(for $n=1$ one $\P^1$ with a self-intersection). The dual graph is an $n$-gon:
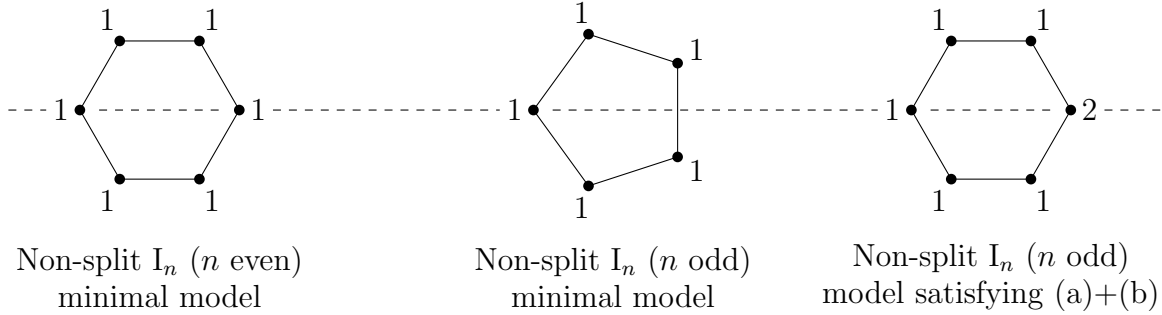
\begin{center}
\begin{tikzpicture}[
    v/.style = {circle, fill=black, inner sep=0pt, minimum size=4pt},
    e/.style = {draw=black},
    l/.style = {inner sep=4pt}
  ]
  \draw[dashed] (-1,2)--(-0.5,2) (0.2,2)--(1.8,2) (2.5,2)--(5.5,2) (6.2,2)--(10.5,2) (11.2,2)--(12.8,2) (13.5,2)--(14.3,2);
  \hexagon{1,1,1,1,1,1}{1}{2}{30pt}        
  \pentagon{1,1,1,1,1}{7}{2}{30pt}         
  \hexagon{2,1,1,1,1,1}{12}{2}{30pt}       
  \node[align=center] at (1,-0.2) {Non-split I$_n$ ($n$ even)\\minimal model};
  \node[align=center] at (7,-0.2) {Non-split I$_n$ ($n$ odd)\\minimal model};
  \node[align=center] at (12,-0.2) {Non-split I$_n$ ($n$ odd)\\model satisfying (a)+(b)};
\end{tikzpicture}
\end{center}
The Galois group acts as a reflection, fixing the identity component (leftmost component in every picture).
When $n$ is odd, it inverts an edge (middle picture), and blowing up the corresponding intersection point
gives an extra component of multiplicity 2 (right picture).
The left and right models satisfy (a)+(b), and the quotient graph $\DG=\D/G$ is
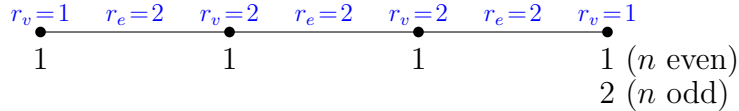
\begin{center}
\begin{tikzpicture}[
    xscale = 2.5, v/.style = {circle, fill=black, inner sep=0pt, minimum size=4pt},
    e/.style = {draw=black}, l/.style = {inner sep=3pt, scale=0.8, blue}, m/.style = {below,inner sep=6pt}]
  \draw[e] (0,0) -- node[l, above] {$r_e\!=\!2$} (1,0)                                       
     -- node[l,above] {$r_e\!=\!2$} (2,0) -- node[l,above] {$r_e\!=\!2$} (3,0);
  \node[v] at (0,0) {}; \node[v] at (1,0) {}; \node[v] at (2,0) {}; \node[v] at (3,0) {};   
  \node[m] at (0,0) {1}; \node[m] at (1,0) {1}; \node[m] at (2,0) {1};                      
  \node[m, align=left, anchor=north west] at (2.87,0.05) {1 ($n$ even)\\[-1pt]2 ($n$ odd)};
  \node[l,above] at (0,0) {$r_v\!=\!1$}; \node[l,above] at (1,0) {$r_v\!=\!2$};             
  \node[l,above] at (2,0) {$r_v\!=\!2$}; \node[l,above] at (3,0) {$r_v\!=\!1$};
  \node at (1.5,-1.3) {Quotient graph $\DG$ for an elliptic curve with non-split multiplicative reduction.};
\end{tikzpicture}
\end{center}
The terms in the theorem are (with trivial $\Psi_J$ as $E(K)\ne\emptyset$)
$$
  \prod_{v\in \VG} m_v^{\deg v - 2} =
    \begin{cases} 1 & n \text{ even} \\ 1/2 & n \text{ odd}, \end{cases} \qquad
  \det\bigl(\langle\cdot,\cdot\rangle\big|_{H_1(\DG,\Z)}\bigr) = 1, \qquad
  \frac{\prod_{e\in \EG} r_e}{\prod_{v\in \VG} r_v} = 2,
$$
multiplying to $|\phi_J(k)|=2$ for $n$ even and $|\phi_J(k)|=1$ for $n$ odd.
If $C$ has \emph{split} multiplicative reduction of type I$_{n}$, then
$\D$ is an $n$-gon with trivial $G$-action, and
$$
  \prod_{v\in \VG} m_v^{\deg v - 2} = 1, \qquad
  \det\bigl(\langle\cdot,\cdot\rangle\big|_{H_1(\DG,\Z)}\bigr) = n, \qquad
  \frac{\prod_{e\in \EG} r_e}{\prod_{v\in \VG} r_v} = 1, \qquad |\phi_J(k)|=n.
$$
\end{example}

\section{History and overview}
\label{shistory}

Tamagawa numbers for abelian varieties were defined by Tate \cite{TatC}, as
the correct `fudge factors' for the Birch--Swinnerton-Dyer
conjecture, and he classified them for elliptic curves in \cite{TatA}.
In higher genus, the question of understanding component groups for Jacobians has a rich history,
all rooted in Raynaud's theorem \cite[\S8]{Ray}:

\subsection{Geometric case}

Suppose $k=\bar k$.
Following Raynaud, consider the maps

\begin{equation}
\label{eqnraynaud}
\begin{tikzpicture}[scale=1.2, commutative diagrams/every diagram, inj/.style={{Hooks[right]}->}, surj/.style={->>}]
\node (ZV1)  at (0,0) {$\Z^{V}$};
\node (ZV2)  at (4,0) {$\Z^{V}$};
\node (Z)    at (7,0) {$\Z$.};
\path[commutative diagrams/.cd, every arrow, every label]
(ZV1) edge node[swap] {$v\mapsto\sum_{w} (v\cdot w)w$} node {$\alpha$} (ZV2)  
(ZV2) edge node[swap] {$v\mapsto m_v$} node {$\beta$} (Z);     
\end{tikzpicture}
\end{equation}

\noindent
The fibre relation $\cC_k\cdot\text{anything}=0$ implies $\im\alpha\subseteq\ker\beta$, and
there is a canonical short exact sequence of $G$-modules%
\footnote{Raynaud uses actual (unnormalised) multiplicities
rather than $\mathbf{m}$, but this does not change $\ker\beta$.}
(Raynaud {\cite[9.6.1]{BLR}}, Bosch--Liu {\cite[Thm.~1.1, Cor.~1.10]{BL}})
\begin{equation}\label{raynaudeq}
  0 \longrightarrow \im\alpha \longrightarrow \ker\beta
    \longrightarrow \phi_J(\bar k) \longrightarrow 0.
\end{equation}
So the component group is an invariant of the intersection matrix;
geometric genera of components (`abelian part') play no role, and understanding
$$
  \phi_J(\bar k)=\frac{\ker\beta}{\im\alpha}=(\coker \alpha)_{\tors}
$$
reduces to linear algebra.

In the purely unipotent case, when the dual graph $\D$ is a tree,
Theorem \ref{etnc} reduces to the elegant formula of Lorenzini \cite[Cor.~2.3]{Lo}, \cite[Cor.~1.5]{LoComp}
(see also \cite{Padma})
$$
  |\phi_J(\bar k) | \>=\>
    \prod_{v\in V} m_v^{\deg v - 2}.
$$
Note that even in this `simplest' case, no closed-form description of the group $\phi_J(\bar k)$ is known,
only for its order. It seems generally hard, even for semistable
curves; see e.g. the summary in \cite{Alar} in the language of graphs and their Laplacians.

In the opposite purely toric case ($C$ semistable), the formula becomes
\begin{equation}\label{itorfor}
  |\phi_J(\bar k) | \>=\>    \det\bigl(\langle\cdot,\cdot\rangle\big|_{H_1(\D,\Z)}\bigr), \qquad \langle e,e'\rangle = \delta_{ee'}.
\end{equation}
This is well-known, and goes back essentially to Kirchhoff's 1847 paper,
where he introduced fundamental cycles for electrical networks \cite{Kir47}.
Here $\D$
is an ordinary unweighted graph ($m_v=1$), and the determinant on the right has been studied
under many different names, e.g. Jacobian, component group, chip-firing group, sandpile group, number of spanning trees,
complexity; see \cite{Lo08} for a survey, and also \cite{Betts}, \cite{HN}, \cite{HNbase}, \cite{MZ}.

One can prove \eqref{itorfor}
via the ``flow-cut decomposition'': the lattice $\Z^E$ splits
orthogonally over $\Q$ into the cycle space $H_1(\D,\Z)$ and a
complementary ``cut space'' generated by the boundaries of vertices,
and $\phi_J(\bar k)$ can be recovered from here as a group \cite[\S9.6]{BLR}, \cite{Bacher}.
We extend this construction to the non-semistable case (when $k=\bar k$):

\namedthm{B}
[Weighted flow-cut = Thm.~\ref{flowcut}]
\label{iflowcut}
Choose an orientation of the edges of $\D$.
Define the length pairing $\langle e,e'\rangle=\delta_{ee'}s_e$ with $s_{\e vw}=1/m_vm_w$.
There is a lift $\lambda\colon\Z^V\to\Z^E$ whose image is orthogonal to $H_1(\D,\Z)$
with respect to $\lara$, and an exact sequence
$$
  0 \>\longrightarrow\> \phi_J(\bar k)
    \>\longrightarrow\>
    \frac{\Z^E}{H_1(\D,\Z)+\lambda(\Z^V)}
    \>\xrightarrow{\e{v}{w}\mapsto 1_v-1_w}\>
    \bigoplus_{v\in V}\frac{\Z}{m_v\Z} \>\longrightarrow\> 0.
$$
\endnamedthm

\noindent
It is then easy to deduce Theorem~\ref{etnc} when $k=\bar k$ by analysing indices, see Theorem~\ref{TNC}.

\subsection{Arithmetic case}

Now suppose $k$ is any perfect field and $G=\Gal(\bar k/k)$.
Taking Galois invariants in Raynaud's sequence \eqref{raynaudeq} yields an exact sequence
$$
  0 \lar \frac{\ker(\beta^G)}{\im(\alpha^G)} \lar \phi_J(k)
    \lar \ker\left(\psi\colon H^1(G,\im\alpha)\to H^1(G,\ker\beta)\right) \lar 0.
$$
Following Bosch--Liu \cite[\S1]{BL}, we compute $|\phi_J(k)|$ from the other two terms.
Let us focus on $\ker(\beta^G)/\im(\alpha^G)$ now, postponing the `cohomological term'
$\ker\psi$ to \S\ref{sscoh}.
Write

\begin{center}
\begin{tabular}{lll@{\qquad}l}
$\DG$    &=& $(\VG,\EG)$ quotient graph: $\VG =V/G$, $\EG=E/G$ sets of orbits, \\
$r_v, r_e$ &=& orbit sizes for $v\in \VG, e\in \EG$,\\
$\Vlift{v}$ &=& $\sum_{v' \in v} v'$ the natural lift of $v \in \VG$ to $\Z^V$,\\
$m'$       &=& $\gcd_{v \in \VG}(r_v m_v)$.
\end{tabular}
\end{center}

The map $\alpha^G$ takes values in $\Z^{\VG}$ \cite[Lem.~1.4]{BL}, and
$\alpha^G$ and $\beta^G$ are as follows:

\begin{center}
\begin{tikzpicture}[scale=1.2, commutative diagrams/every diagram, inj/.style={{Hooks[right]}->}, surj/.style={->>}]
\node (ZV1)  at (0,0) {$\Z^{\VG}$};
\node (ZV2)  at (4,0) {$\Z^{\VG}$};
\node (Z)    at (7,0) {$\Z$.};
\path[commutative diagrams/.cd, every arrow, every label]
(ZV1) edge node[swap] {$v\mapsto\sum_{w\in \VG} \frac{1}{r_w}(\Vlift{v}\cdot\Vlift{w})w$} node {$\alpha^G$} (ZV2)  
(ZV2) edge node[swap] {$v\mapsto r_v m_v$} node {$\beta^G$} (Z);     
\end{tikzpicture}
\end{center}

\noindent
This is very similar to the $k=\bar k$ case, except that $\alpha^G$ is not symmetric,
so Theorem \ref{iflowcut} does not apply directly. We can symmetrise it as follows:
$$
  \alpha^{\sym}\colon\Z^{\VG}\lar \Z^{\VG}, \qquad
  v\mapsto\sum_{w\in \VG} \frac{N}{r_vr_w}(\Vlift{v}\cdot\Vlift{w})w \qquad
  (N=\mathop{\rm lcm}\limits_{v\in\VG} r_v).
$$
From elementary linear algebra (Lemma \ref{BLsym}), there is an exact sequence
\begin{equation}
  0 \lar \frac{\ker(\beta^G)}{\im(\alpha^G)} \lar
    \coker(\alpha^{\sym})_{\tors} \lar
    \bigoplus_{v\in \VG} \Z/\tfrac{N}{r_v}\Z \lar
    \Z/\tfrac{N}{m'}\Z \lar 0,
\end{equation}
so it suffices to compute $\coker(\alpha^{\sym})_{\tors}$.
This has again a graph-theoretic interpretation, and using standard reduction steps
for weighted graphs that stem from electrical networks (see \S\ref{stamgraph}),
we compute it and get the arithmetic part
of Theorem \ref{etnc} and the term $\gcd_{v\in \VG} (r_v m_v)$, see Theorem \ref{tarith}.

For an alternative description of $|\ker\beta^G/\im\alpha^G|$ via spanning trees,
see Srinivasan \cite{Padma}.

\begin{example}
\label{IV*---IV*}
Suppose $k=\bar k$, and take the genus 4 reduction family
$\redtype{IV^*\e(\hbox{\tiny $m$})\e(\hbox{\tiny $n$})\e(\hbox{\tiny $r$})IV^*}$,
in the notation of \cite{types}. Its special fibre is made of two Kodaira types \IVS{}.
It has two principal components of multiplicity 3 linked with chains of $\P^1$s of depth $m$, $n$ and $r$ (left picture):

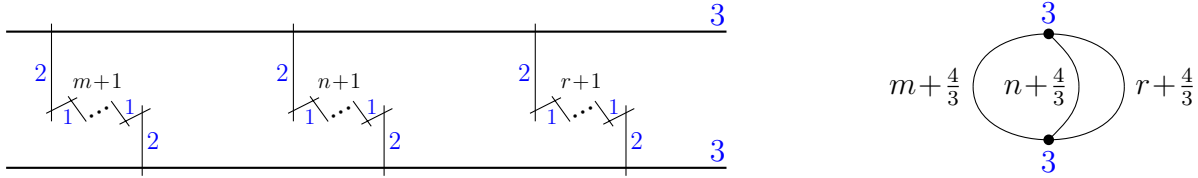
\begin{center}
\begin{tikzpicture}[o1/.style={blue,scale=0.7,inner sep=1,anchor=north}, o2/.style={blue,scale=0.7,inner sep=1,anchor=south}, o3/.style={thick,shorten >=2}, o4/.style={inner sep=2pt,blue,above left,scale=1}, o5/.style={auto,inner sep=9,scale=0.7,anchor=south}, o6/.style={shorten <=-3pt, shorten >=-3pt}, o7/.style={inner sep=2pt,scale=0.8,blue,left}, o8/.style={inner sep=2pt,scale=0.8,blue,right},xscale=1,yscale=0.9] \draw[o3] (-3/5,0)--(9,0) node[o4] {3}; \draw[o6] (6/5,0)--node[o8] {2} (6/5,4/5); \draw (6/5,4/5) ++(0.11,0.07)--++(-0.41,-0.23)++(0.15,0.12) ++(-0.01,0) node[o2]{1} ++(0.01,0) ++(-0.02,-0.15)--++(-0.24,0.37)++(-0.06,-0.11) node{$\cdot$} ++(-0.1,-0.06) node{$\cdot$} ++(-0.1,-0.06) node{$\cdot$}++(0.08,0.15) node[o5]{$m\!+\!1$} ++(-0.15,-0.22)--++(-0.24,0.37)++(-0.03,-0.16) ++(0.03,0) node[o1]{1} ++(-0.03,0) ++(0.03,0.16)++(0.12,-0.02)--++(-0.41,-0.23); \draw[o6] (0,4/5)--node[o7] {2} (0,10/5); \draw[o6] (22/5,0)--node[o8] {2} (22/5,4/5); \draw (22/5,4/5) ++(0.11,0.07)--++(-0.41,-0.23)++(0.15,0.12) ++(-0.01,0) node[o2]{1} ++(0.01,0) ++(-0.02,-0.15)--++(-0.24,0.37)++(-0.06,-0.11) node{$\cdot$} ++(-0.1,-0.06) node{$\cdot$} ++(-0.1,-0.06) node{$\cdot$}++(0.08,0.15) node[o5]{$n\!+\!1$} ++(-0.15,-0.22)--++(-0.24,0.37)++(-0.03,-0.16) ++(0.03,0) node[o1]{1} ++(-0.03,0) ++(0.03,0.16)++(0.12,-0.02)--++(-0.41,-0.23); \draw[o6] (16/5,4/5)--node[o7] {2} (16/5,10/5); \draw[o6] (38/5,0)--node[o8] {2} (38/5,4/5); \draw (38/5,4/5) ++(0.11,0.07)--++(-0.41,-0.23)++(0.15,0.12) ++(-0.01,0) node[o2]{1} ++(0.01,0) ++(-0.02,-0.15)--++(-0.24,0.37)++(-0.06,-0.11) node{$\cdot$} ++(-0.1,-0.06) node{$\cdot$} ++(-0.1,-0.06) node{$\cdot$}++(0.08,0.15) node[o5]{$r\!+\!1$} ++(-0.15,-0.22)--++(-0.24,0.37)++(-0.03,-0.16) ++(0.03,0) node[o1]{1} ++(-0.03,0) ++(0.03,0.16)++(0.12,-0.02)--++(-0.41,-0.23); \draw[o6] (32/5,4/5)--node[o7] {2} (32/5,10/5); \draw[o3] (-3/5,10/5)--(9,10/5) node[o4] {3}; \end{tikzpicture}
\qquad\qquad
\raise-5pt\hbox{
\begin{tikzpicture}[
    v/.style = {circle, fill=black, inner sep=0pt, minimum size=4pt},
    e/.style = {draw=black},
    l/.style = {inner sep=4pt}
  ]
  \coordinate (P1) at (0,-0.7);               
  \coordinate (P2) at (0,0.7);
  \coordinate (C1) at (-1,0);                 
  \coordinate (C2) at (0.4,0);
  \coordinate (C3) at (1,0);
  \draw[e] (P1) to[out=180,in=270] (C1);      
  \draw[e] (C1) to[out=90,in=180] (P2);
  \draw[e] (P1) to[out=0,in=270] (C3);
  \draw[e] (C3) to[out=90,in=0] (P2);
  \draw[e] (P1) to[out=45,in=270] (C2);       
  \draw[e] (C2) to[out=90,in=-45] (P2);
  \node[v] at (P1) {}; \node[v] at (P2) {};   
  \node[l,below,blue] at (P1) {3};            
  \node[l,above,blue] at (P2) {3};
  \node[l,left]  at (C1) {$m\!+\!\frac 43$};  
  \node[l,left]  at (C2) {$n\!+\!\frac 43$};
  \node[l,right] at (C3) {$r\!+\!\frac 43$};
\end{tikzpicture}
}\\
Special fibre of $\redtype{IV^*\e(\hbox{\tiny $m$})\e(\hbox{\tiny $n$})\e(\hbox{\tiny $r$})IV^*}$ and the associated weighted graph as a metric space.
\end{center}

The total length of the chains of $\P^1$s with respect to $\lara$ is
$$
  \tfrac{1}{3\cdot 2}
  + \tfrac{1}{2\cdot 1}
  + \underbrace{\tfrac{1}{1\cdot 1} + \cdots + \tfrac{1}{1\cdot 1}}_{m\text{ intersection points}}
  + \tfrac{1}{1\cdot 2}
  + \tfrac{1}{2\cdot 3}
  \>=\> m + \tfrac{4}{3},
$$
and, respectively, $n+\tfrac{4}{3}$ and $r+\tfrac{4}{3}$. In the obvious basis for homology,
$$
  \det\bigl(\langle\cdot,\cdot\rangle\big|_{H_1(\DG,\Z)}\bigr) =
  \Biggl|
  \begin{matrix}
  (m\!+\!\frac43) \!+\! (n\!+\!\frac43)\!\!\!\! & -(n\!+\!\frac43) \\[3pt]
  -(n\!+\!\frac43) & (n\!+\!\frac43) \!+\! (r\!+\!\frac43) \cr
  \end{matrix}
  \Biggr| = \frac{3(mn\!+\!mr\!+\!nr) \!+\! 8(m\!+\!n\!+\!r) \!+\! 16}{3}.
$$
The unipotent part is
$$
  \prod_{v\in \VG} m_v^{\deg v - 2} \>=\> 3^{3-2} \cdot 3^{3-2} \cdot \prod_{v\text{ in chains}} m_v^{\,0} \>=\> 9,
$$
and the Tamagawa number is the product
$
  |\phi_J(\bar k)| \>=\> 9(mn\!+\!mr\!+\!nr) + 24(m\!+\!n\!+\!r) + 48.
$
\end{example}

\begin{remark}
As in the example above, in minimal models with normal crossings, Tamagawa numbers respect
the distinction between principal components and chains of $\P^1$s. In the terminology of \cite{types}:

\begin{itemize}
\item[(i)]
Interior chain components play little role: in the unipotent and toric parts
of Theorem \ref{etnc} they only enter through total length of chains.
\item[(ii)]
It is not hard to deduce from (i) that for any family of reduction types, its Tamagawa number
is a quadratic form in the depths of its chains.
\item[(iii)]
Similarly, in the arithmetic case, the interior chain components and their edges have
the same orbit sizes, so $r_e/r_v$ cancels for them as well.
\item[(iv)]
Thus, a natural framework for the Tamagawa numbers is weighted graphs that are insensitive to degree
2 vertices (cf. first picture on the right in Example \ref{IV*---IV*}).
This is the formalism summarised in \S\ref{stamgraph}, and used to prove Theorem \ref{tarith}.
\end{itemize}

\end{remark}

\subsection{Cohomological part}
\label{sscoh}

The most intricate term is the kernel of
$$
  \psi \colon H^1(G, \im\alpha) \to H^1(G, \ker\beta).
$$
When $G$ is procyclic, Bosch--Liu \cite[Lemmas 1.11, 1.19, Thm.~1.17]{BL}
compute both groups and the map $\psi$ explicitly, using cyclic homology. If
$(d_v)_v$ are the actual (non-normalised\footnote{%
In terms of normalised multiplicities $m_v = d_v/c$, one has $d = \gcd d_v = c$ and $d' = c\,m'$.})
geometric multiplicities, $d=\gcd_{v\in\VG}(d_v)$,
$d' = \gcd_{v\in\VG}(r_v d_v)$, they show
\begin{itemize}
\item[i.] $H^1(G, \ker\beta) \>\cong\> d\Z/d'\Z\>\> (\>\iso \Z/m'\Z)$,
\item[ii.] $H^1(G, \im\alpha)  \>\cong\> d\Z/d'\Z\>\> (\>\iso \Z/m'\Z)$,
\item[iii.] the image of $\psi$ has order $q=1$ if $d' \mid (\genus(C/K) - 1)$
  and order $q=2$ otherwise.
\end{itemize}

\noindent
In \cite[Remark~1.18]{BL}, they ask what happens for general $G$, and we show:

\namedthm{C}[=Theorem \ref{psifor}]\label{ipsiJ} In the notation of Definition~\ref{defpsi},
$$
  \ker\bigl(\psi\colon H^1(G,\im\alpha)\to H^1(G,\ker\beta)\bigr) \>\iso\> \Psi_J.
$$
\endnamedthm

It turns out that $H^1(G, \ker\beta)\iso\Z/m'\Z$ in general, and
$H^1(G, \im\alpha)$ can be seen as a subgroup of $\Gvee = \Hom(G, \Q/\Z)$ (Theorem~\ref{kerpsi}).
We find an explicit cocycle formula for~$\psi$,
which Theorem \ref{psifor} rewrites as a sum of parities.
The dichotomy $|\im\psi| \in \{1,2\}$ still holds as in the cyclic case,
but $\ker\psi$ may be non-cyclic, and can have different order.

\begin{example}
Let $\pi$ be a uniformiser of $\cO_K$ and $f(x)\in\cO_K[x]$ a quartic such that
$f \bmod \pi$ has distinct roots $\alpha_1,...,\alpha_4\in \bar k^\times$.
By \cite[Thm.~3.13]{newton}, the equation of a genus 3 curve
$$
  C: \pi y^4 = f(x)
$$
is `$\Delta_v$-regular', of reduction type \redtype{4^{1,1,1,1}} in the notation of \cite{types}:
the special fibre $\cC_k$ has one principal component of multiplicity 4 meeting four
$\P^1$s of multiplicity 1,%
$$
\begin{tikzpicture}[o1/.style={shorten <=-3pt}, o2/.style={inner sep=2pt,scale=0.8,blue,right},xscale=1,yscale=0.9]
\draw[thick,shorten >=2] (-1/3,0)--(121/30,0) node[inner sep=2pt,blue,above left,scale=0.8] {4};
\draw[o1] (0,0)--node[o2] {1} (0,3/5);
\draw[o1] (4/5,0)--node[o2] {1} (4/5,3/5);
\draw[o1] (8/5,0)--node[o2] {1} (8/5,3/5);
\draw[o1] (12/5,0)--node[o2] {1} (12/5,3/5);
\end{tikzpicture}
$$
and these four $\P^1$s are in bijection with the roots $\alpha_i$, as a $G$-set.
This $G$-set can be arbitrary, if e.g. $k=\Q$, and we obtain the following possibilities from \ref{kerpsi}
(with $q$ defined as in iii):

\begin{center}
\begin{tabular}{ccccccccc}
\multicolumn{2}{l}{Image of $G$} & Generators & $m'$ & $H^1(G,\im\alpha)$ & $H^1(G,\ker\beta)$ & $\Psi_J=\ker(\psi)$ & $m'/q$ \\
\hline
$C_1$     & cyclic & $1$                      & 1 & $C_1$   & $C_1$ & $C_1$  & 1 \\
$C_2$     & cyclic & $(12)(34)$               & 2 & $C_2$   & $C_2$ & $C_2$  & 2 \\
$C_2$     & cyclic & $(12)$                   & 1 & $C_1$   & $C_1$ & $C_1$  & 1 \\
$C_3$     & cyclic & $(123)$                  & 1 & $C_1$   & $C_1$ & $C_1$  & 1 \\
$C_4$     & cyclic & $(1234)$                 & 4 & $C_4$   & $C_4$ & $C_2$  & 2 \\
$V_4$     &        & $(12)(34),(13)(24)$      & 4 & $V_4$   & $C_4$ & $V_4$  & 2 \\
$V_4$     &        & $(12),(34)$              & 2 & $C_1$   & $C_2$ & $C_1$  & 2 \\
$S_3$     &        & $(123),(12)$             & 1 & $C_1$   & $C_1$ & $C_1$  & 1 \\
$D_4$     &        & $(1234),(13)$            & 4 & $C_2$   & $C_4$ & $C_2$  & 2 \\
$A_4$     &        & $(123),(12)(34)$         & 4 & $C_1$   & $C_4$ & $C_1$  & 2 \\
$S_4$     &        & $(1234),(12)$            & 4 & $C_1$   & $C_4$ & $C_1$  & 2 \\
\end{tabular}\\[2pt]
Example. Cohomological term for the genus 3 quartic $\pi y^4 = f(x)$ of type \redtype{4^{1,1,1,1}}\\
for a given action of $G=\Gal(\bar k/k)$ on the four distinct roots of $f\bmod\pi$.
\end{center}
When $G$ acts on $\D$ through a cyclic group (first five cases), $|\ker\psi|=m'/q$ as proved by Bosch--Liu. Otherwise, $\ker\psi$ is not always
cyclic, and can have order both smaller and larger than $m'/q$, as the table illustrates.
\end{example}

\begin{remark}
If $K$ is a global field with a valuation $v$, then
$k$ is finite and $G$ is procyclic.
Here $q=2$ for $C/K_v$ is equivalent to the curve being `deficient' in
the sense of Poonen--Stoll \cite[\S8]{PS} (i.e. $\Pic^{g-1}C(K_v)=\emptyset$), a condition that
they relate to $\Sha(C/K)$ being of square order (if finite). The description of $\ker\psi=\Psi_J$
in Theorem \ref{ipsiJ} can be viewed as an analogue of deficiency when
$k$ is a perfect but not necessarily finite.
\end{remark}

\section{Weighted flow-cut}
\label{sgeomgen}

In this section we prove Theorem \ref{iflowcut} and deduce Theorem \ref{etnc} in the geometric case
(no action of $G$). This is a question about the intersection pairing that can be phrased
purely in terms of matrices.

Let $M\in M_n(\Z)$ be any symmetric matrix with
non-negative off-diagonal entries and kernel of rank 1, generated by a primitive vector
$\mathbf{m}\in\N^n$.
Let $V$ index the rows and columns of $M$, and write
$v\cdot w=M_{vw}$, the corresponding entry of $M$.
Let $\alpha$ and $\beta$ be Raynaud's maps given by \eqref{eqnraynaud}.
Write $\D=\D(M)$ for the multigraph with vertex set $V$,
multiplicity $m_v$ at each $v\in V$, and $M_{vw}$ edges between
$v$ and $w$ for each pair $v\ne w$ (no loops).
Since $\ker M$ has rank $1$, it follows that $\D$ is connected.
Finally, write
\begin{equation}\label{eqbasesgen}
  \Z^V=\bigoplus_{v\in V}\Z v \quad\subset\quad
  \bigoplus_{v\in V} \Z \frac{v}{m_v} = \Z^{V/m}, \qquad
  \Z^E=\bigoplus_{e\in E}\Z e.
\end{equation}

\comment
\subsection{Setup}

Throughout this section, $M\in M_n(\Z)$ is a symmetric matrix with
non-negative off-diagonal entries, whose kernel has rank 1, generated by a primitive vector
$\mathbf{m}\in\N^n$.
The associated \emph{weighted graph} $\D=\D(M,\mathbf{m})=(V,E,(m_v)_{v\in V})$ is:
\begin{itemize}
  \item $|V|=n$, and each vertex $v$ carries the \emph{multiplicity} $m_v$, so that $\mathbf{m}=(m_v)_v$.
  \item For distinct $v,w\in V$, the number of edges between $v$ and $w$
        equals $M_{vw}\ge 0$.
  \item The \emph{degree} $\deg v$ of a vertex $v$ is the number of
        edge-ends at $v$.
  \item Choose an orientation for every edge, and write $\e vw$ (or $e$) for an edge from $v$ to $w$.
  \item The \emph{length } of an edge $e=\e vw\in E$ is $s_e=\tfrac{1}{m_v m_w}$,
     and we define the \emph{length pairing} on $\Z^E$ by
     $\langle e,e'\rangle = s_e\,\delta_{ee'}$.
\end{itemize}
The \emph{intersection pairing} is the symmetric bilinear form
$\cdot\colon \Z^V\times\Z^V\to\Z$ given by $M$, so $v\cdot w=M_{vw}$.
By assumption, $\mathbf{m}=(m_v)_{v\in V}$ is in the kernel. Thus, for all $v\in V$,
$$
  \sum_{w\in V} m_w\,(v\cdot w) \>=\> 0 \qquad\text{and}\qquad v\cdot v = -\frac1{m_v}\sum_{w\ne v} m_w(v\cdot w)\in\Z_{\le 0}.
$$
The \emph{component group} of $M$ is
$$
  \coker(M)_{\tors} = \coker(M)_{\tors}.
$$

\endcomment

\begin{definition}[Weighted lift $\lambda$]
Choose a direction for every edge $e$ of $\G$, so edges will be denoted by $\e{v}{w}$.
Define maps
$$
\begin{array}{llllllll}
  \partial:  & \Z^E  & \to & \Z^{V} && \multicolumn{2}{l}{\text{(incidence map)}}\cr
           & \e vw & \mapsto & v-w\Bcr
  \delta:  & \Z^E  & \to & \Z^{V/m} && \multicolumn{2}{l}{\text{(weighted incidence map)}}\cr
           & \e vw & \mapsto & \frac{v}{m_v} - \frac{w}{m_w}\Bcr
  \lambda: & \Z^V  & \to & \Z^E & \text{where} &
                \e{v}{w_i} & \text{are all edges from $v$,}\Tcr
           & v   & \mapsto & \sum_i m_{w_i} \e{v}{w_i} - \sum_j m_{z_j} \e{z_j}{v}&&
                \e{z_j}{v} & \text{are all edges to $v$.}
\end{array}
$$
All three depend on the orientation, but $\delta \circ \lambda$ and $\partial \circ \lambda$ do not.
\end{definition}

\begin{proposition}
\label{bigdiagram}
The following diagram commutes, and its bottom row is exact:

\smallskip

\begin{tikzpicture}[scale=1.2, commutative diagrams/every diagram, inj/.style={{Hooks[right]}->}, surj/.style={->>}]
\node (ZV1)  at (0,2) {$\Z^V$};
\node (ZV2)  at (3,2) {$\Z^V$};
\node (ZE)   at (0,0) {$\Z^E$};
\node (Zsum) at (3,0.8) {$\Z^{V/m}$};
\node (ZV3)  at (6,0) {$\Z^V$};
\node (Z)    at (9,0) {$\Z$.};
\node (H1)   at (-3,0) {$H_1(\D,\Z)$};
\path[commutative diagrams/.cd, every arrow, every label]
(ZV1) edge node {$v\mapsto\sum_w (v\cdot w)w$} node[swap] {$\alpha$} (ZV2)  
(ZV2) edge[bend left=10] node {$v\mapsto m_v$} node[swap] {$\beta$} (Z)     
(ZV1) edge node[swap] {$-\lambda$} (ZE)                                     
(ZV2) edge[inj] node {$v\mapsto v$} node[swap] {$\mu$} (Zsum)               
(ZE) edge[bend left=10] node[sloped] {$\e vw \,\mapsto\, \frac{v}{m_v} - \frac{w}{m_w}$}    
  node[sloped, swap] {$\delta$} (Zsum)
(Zsum) edge[bend left=10] node[sloped] {$\frac{v}{m_v}\mapsto v$} node[sloped,swap]{$\sigma\>\>(\iso)$} (ZV3)   
(H1) edge[inj] node[swap] {$\iota$} (ZE)
(ZE)  edge[] node[] {$\e vw \,\mapsto\, v-w$} node[swap] {$\partial$} (ZV3)
(ZV3) edge[surj] node{$v\mapsto 1$} node[swap] {$\triv$} (Z);
\end{tikzpicture}

\noindent
In particular, $-\lambda$ is a lift of the intersection pairing map $\alpha$ to the edges: for all $v\in V$,
$$
  (-\delta\circ\lambda)(v) \>=\> \sum_{w\in V} (v\cdot w)\,w \qquad \in \Z^{V/m}.
$$
\end{proposition}

\begin{proof}
The bottom row defines $H_1(\D,\Z)$, and its exactness is standard. The only non-trivial
claim is $-\delta\circ\lambda=\mu\circ\alpha$. Pick $v\in V$.
As $(\delta\circ\lambda)(v)$ is unchanged if an edge direction is reversed,
we can direct all edges at $v$ away from $v$ for simplicity (no $z_j$'s). Then
$$
  (\delta\circ\lambda)(v) = \delta(\sum_i m_{w_i} \e{v}{w_i})
   = \sum_i m_{w_i} (\frac{v}{m_v}-\frac{w_i}{m_{w_i}}) = \frac{\sum_i m_{w_i}}{m_v} v - \sum_i w_i
   = -(v\cdot v)v-\sum_{w\ne v}(v\cdot w)w.
$$
\end{proof}

\begin{theorem}[Weighted flow-cut]
\label{flowcut}
Define the length pairing on $\Z^E$ by $\langle e,e'\rangle=\delta_{ee'}s_e$ with $s_{\e vw}=1/m_vm_w$.
In the setup of this section,
\begin{itemize}
\item[(1)] $H_1(\D,\Z)$ is the orthogonal complement of $\lambda(\Z^V)$ in $\Z^E$
with respect to $\lara$:
$$
  H_1(\D,\Z) \>=\> \bigl\{\,y\in\Z^E \,\big|\, \langle\lambda(x),y\rangle=0
    \text{ for all } x\in\Z^V\,\bigr\}.
$$

\item[(2)] There is a short exact sequence
$$
  0 \>\longrightarrow\> \coker(M)_{\tors}
    \>\longrightarrow\>
    \frac{\Z^E}{H_1(\D,\Z)+\lambda(\Z^V)}
    \>\xrightarrow{\>\>\gamma\>\>}\>
    \bigoplus_{v\in V}\frac{\Z}{m_v\Z} \>\longrightarrow\> 0,
$$
where $\gamma$ sends an edge $\e{v}{w}$ to the class of $\partial(\e{v}{w})=v-w\in\Z^V$
in the direct sum.
\end{itemize}
\end{theorem}

\begin{proof}
For a map $f$ between the spaces $\Z^V$, $\Z^E$, $\Z^{V/m}$ write $M_f$ for its matrix
in the bases \eqref{eqbasesgen}. Let $Q=\diag(s_e)$ be the matrix of $\lara$
on $\Z^E$, and $D=\diag(m_v)$.

(1)
From the definition of $\lambda$ and~\ref{bigdiagram},
$$
  M_\lambda \>=\> Q^{-1} M_\partial^T\, D^{-1},
  \qquad M_\lambda^T Q M_\lambda \>=\> -M.
$$
Hence for $x\in\Z^V$ and $y\in\Z^E$,
$$
  \langle\lambda(x),y\rangle \>=\> (M_\lambda x)^T Q\,y \>=\> x^T M_\lambda^T Q\,y
    \>=\> x^T D^{-1} M_\partial\,y \>=\> x^T D^{-1}\,\partial(y).
$$
Since $D^{-1}$ is invertible over $\Q$, this vanishes for all $x\in\Z^V$
if and only if $\partial(y)=0$, that is $y\in H_1(\D,\Z)$.

(2)
From the commutative diagram in \ref{bigdiagram}, $\sigma\circ\delta=\partial$ and $\sigma\circ\mu=D$,
so $\partial\circ\lambda = \sigma\circ(\delta\circ\lambda) = -\sigma\circ(\mu\circ\alpha) = -D\alpha$.
The bottom row is exact, so $\partial\colon\Z^E\to\ker(\triv)$ is surjective with
kernel $\im\iota = H_1(\D,\Z)$.
For $e\in\Z^E$ with $\partial(e)\in\im(D\alpha)$, write $\partial(e)=-\partial(\lambda(u))$;
then $e+\lambda(u)\in\ker\partial = H_1(\D,\Z)$, so $\partial^{-1}(\im D\alpha)=H_1(\D,\Z)+\im\lambda$.
Thus $\partial$ induces an isomorphism
$$
  \Z^E\bigl/(H_1(\D,\Z)+\im\lambda) \>\xrightarrow{\>\sim\>} \ker(\triv)/\im(D\alpha).
$$
Since $\mathbf{m}$ spans the left kernel of $D\alpha=\diag(m_v)\cdot\alpha$,
the map $\triv$ realises the free part of $\coker(D\alpha)$,
so $\ker(\triv)/\im(D\alpha) = \coker(D\alpha)_{\rm tors}$.

Now apply Lemma~\ref{DMfor} with $M=M_\alpha$, $D=D$, $\mathbf{l}=\mathbf{m}$,
$R=\prod_{v\in V}m_v$, and $\hat{d}_v = R/m_v \cdot m_v = R$ for all $v$.
Thus $\hat{d}=1$, and the map $\bigoplus_{v\in V}\Z/m_v\Z\to\Z/\hat{d}\Z=0$ is trivial.
Then \ref{DMfor}(4) yields the required exact sequence, with $\gamma$ as claimed.
\end{proof}

\begin{theorem}
\label{TNC}
Let $M\in M_n(\Z)$ be a symmetric matrix with
non-negative off-diagonal entries and $\ker M=\Z\mathbf{m}$
with $\mathbf{m}\in\N^n$. With $\D=(V,E)$ and $\lara$ as above,
$$
  |\coker(M)_{\tors}| \>\>=\>\>
    \det\bigl(\langle\cdot,\cdot\rangle_{H_1(\D,\Z)}\bigr) \>\cdot\>
    \prod_{v\in V} m_v^{\deg v-2}.
$$
\end{theorem}

\begin{proof}
Taking orders in the exact sequence of Theorem~\ref{flowcut} (2), we get
$$
  |\coker(M)_{\tors}|\>\cdot\>\prod_{v\in V}m_v
   \>=\>
  \Bigl|\frac{\Z^E}{H_1(\D,\Z)+\lambda(\Z^V)}\Bigr|.
$$
Keeping the notation from the proof of \ref{flowcut} (1), recall that
$$
  M_\lambda \>=\> Q^{-1} M_\partial^T\, D^{-1} \>=\> Q^{-1} M_\delta^T,
  \qquad M_\lambda^T Q M_\lambda \>=\> -M.
$$
Since $Q$ is positive definite, $\ker\lambda = \ker(M_\lambda^T Q M_\lambda) = \ker M = \Z\mathbf{m}$.
Extend $\mathbf{m}$ to a $\Z$-basis $\{\mathbf{m},x_1,\ldots,x_{n-1}\}$ of $\Z^V$,
and let $L = \bigoplus_i \Z x_i$, so $\Z^V = \Z\mathbf{m} \oplus L$.
Then $\lambda|_L\colon L\to\Z^E$ is injective with image $\im\lambda$,
and by \ref{pairings}(2),
$$
  \det\bigl(\lara|_{\im\lambda}\bigr)
  \>=\> \det\bigl(M_{\lambda|_L}^T\,Q\,M_{\lambda|_L}\bigr)
  \>=\> \det(-M_L),
$$
where $M_L$ is the $(n{-}1)\times(n{-}1)$ restriction of $M$ to $L$.
In the basis $\{\mathbf{m},x_1,\ldots,x_{n-1}\}$, $M$ has block form
$\smallmatrix000{M_L}$, so
$\coker(-M)_{\tors}\cong\coker(-M_L)$ has order $|\det(-M_L)|$. Hence
$$
  \det\bigl(\lara|_{\im\lambda}\bigr) \>=\> |\coker(-M)_{\tors}| \>=\> |\coker(M)_{\tors}|.
$$
Next, on the full edge space $\Z^E$,
$$
  \det\bigl(\lara\big|_{\Z^E}\bigr) = \det Q = \prod_{\e vw\in E} m_v^{-1} m_w^{-1}
  = \prod_{v\in V}m_v^{-\deg v}.
$$
Finally, $\im\lambda$ and $\im \iota$ are orthogonal for $\lara$ by Theorem \ref{flowcut} (1), so
$$
  \det\bigl(\lara\big|_{\Z^E}\bigr)\cdot\Bigl|\frac{\Z^E}{H_1(\D,\Z)+\lambda(\Z^V)}\Bigr|^2
  = \det\bigl(\lara\big|_{H_1(\D,\Z)}\bigr)\cdot\det\bigl(\lara\big|_{\im\lambda}\bigr)
$$
by the orthogonal decomposition formula \ref{pairings} (4). This becomes
$$
  \prod_{v\in V} m_v^{-\deg v} \cdot \Bigl(|\coker(M)_{\tors}| \cdot \prod_{v\in V} m_v\Bigr)^2  = \det\bigl(\lara|_{H_1(\D,\Z)} \bigr) \cdot |\coker(M)_{\tors}|,
$$
and simplifies to the claimed formula.
\end{proof}

\section{Arithmetic part}
\label{s:arith}

As in \S\ref{sgeomgen}, let $M\in M_n(\Z)$ be a symmetric matrix such that $\ker M=\Z\mathbf{m}$ with
$\mathbf{m}\in\N^n$. Let $\D=(V,E,(m_v))$ be the associated dual graph; we may think of it as
the \emph{geometric} dual graph of a model $\cC$, over $\bar k$.

Let a group $G$ (e.g.\ $\Gal(\bar k/k)$) act on $\D$ without inversion.
Thus, it acts on $V$ and $E$,
and preserves adjacency and multiplicities. As in \S\ref{sintro}, write

\smallskip
\begin{center}
\begin{tabular}{lll@{\qquad}l}
$\DG$    &=& $(\VG,\EG,(m_{v}))$ quotient graph: $\VG =V/G$, $\EG=E/G$ sets of orbits, \\
$r_v, r_e$ &=& orbit sizes for $v\in \VG, e\in \EG$,\\
$m_v$    &=& normalised multiplicity, constant on $G$-orbits (as in \S\ref{sintro}).\\
$\Vlift{v}$ &=& $\sum_{v' \in v} v'$ the natural lift of $v \in \VG$ to $\Z^V$.\\
$m'$       &=& $\gcd_{v \in \VG}(r_v m_v)$.
\end{tabular}
\end{center}

\noindent
Write $\MBL$ for the matrix of $\alpha^G$.
Set $N = \mathrm{lcm}_{v\in \VG}(r_v)$, and
define $\MSym = \MBL \cdot \mathrm{diag}(N/r_v)$.

\begin{lemma}[Symmetrisation]\label{BLsym}
The matrix $\MSym$ is integer symmetric, and has kernel of rank~1 that contains $(m_v r_v)_{v\in\VG}$.
There is an exact sequence of abelian groups
$$
  0 \to \coker(\MBL)_{\tors} \lar \coker(\MSym)_{\tors} \to \bigoplus_{v\in \VG} \Z/\tfrac{N}{r_v}\Z
    \lar \Z/\tfrac{N}{m'}\Z \to 0.
$$
\end{lemma}

\begin{proof}
From $(\MBL)_{wv}\,r_w = \Vlift{v}\cdot\Vlift{w} = (\MBL)_{vw}\,r_v$ we see that
$\MSym$ is symmetric.
As $\mathbf{m}=(m_v)$ is in the right kernel of $\MBL$, the vector $(m_v r_v)$
is in the kernel of $\MSym$.
Since $\MSym$ is symmetric, $\MSym = \MSym^T = \mathrm{diag}(N/r_v)\cdot\MBL^T$,
and we get the claim from Lemma~\ref{DMfor} applied with $M=\MBL^T$, $D=\mathrm{diag}(N/r_v)$, $\mathbf{l}=\mathbf{m}$, and parameters
$$
  R=\prod_{v\in \VG}\frac{N}{r_v}, \quad
  \hat{d}_v = \frac{R}{N/r_v}\cdot{m_v} = \frac{R}{N}\cdot{m_v\,r_v}, \quad
  \gcd_{v\in \VG}(\hat{d}_v) = \frac{Rm'}{N}, \quad
  \hat{d} = \frac{R}{\gcd(\hat{d}_v)} = \frac N{m'}.
$$
\end{proof}

\comment
\begin{remark}
The short exact sequence can be also written as
$$
  0 \to \coker(\MBL)_{\tors}
    \xrightarrow{\bar{D}}
    \coker(\MSym)_{\tors}
    \xrightarrow{\gamma}
    \ker\left(
      \bigoplus_{v\in \VG} \tfrac{r_v}{N}\Z/\Z
      \xrightarrow{\>\Sigma_{\mathbf{m}}\>}
      \Q/\Z
    \right) \to 0,
$$
where $\Sigma_{\mathbf{m}}(a_v)_{v\in \VG} = \sum_{v\in \VG} m_v\,a_v$
and $\gamma$ sends the class of $(a_v)_{v\in \VG}$ to $\bigl(\frac{r_v\, a_v}{N} \bmod\Z\bigr)_{v\in \VG}$.
\end{remark}
\endcomment

\begin{theorem}[Arithmetic case]
\label{tarith}
In the above notation,
$$
  \Bigl|\frac{\ker(\beta^G)}{\im(\alpha^G)}\Bigr|
  \>=\> \prod_{v\in \VG} m_v^{\deg v - 2} \cdot
      \det\bigl(\langle\cdot,\cdot\rangle\big|_{H_1(\DG,\Z)}\bigr) \cdot
      m' \cdot
      \frac{\prod_{e\in \EG} r_e}{\prod_{v\in \VG} r_v}.
$$
\end{theorem}

\begin{proof}
We use the formalism of weighted graphs and their Tamagawa numbers, see \S\ref{stamgraph}.
Consider the following three weighted graphs. In each case the vertex set is $\VG$,
with multiplicities as in column~2. Each edge orbit $\e vw\in\EG$ contributes the
number of edges from $v$ to $w$ given in column~3, each of length as in column~4.

\begin{center}
\begin{tabular}{l@{\quad}c@{\quad}c@{\quad}c}
  Weighted graph & multiplicity of $v$ & \#edges per orbit $\e vw$ & edge length\\
  \hline
  $M_{\sym}$ graph $\cT_1$ & $m_vr_v/m'$ & $Nr_e/(r_vr_w)$ & $m'^2/(m_vr_v\cdot m_wr_w)$ \\
  merged graph $\cT_2$     & $m_vr_v/m'$ & $1$             & $m'^2/(Nr_em_vm_w)$ \\
  quotient graph $\DG$   & $m_v$       & $1$             & $1/(r_em_vm_w)$ \\
\end{tabular}
\end{center}

\noindent
Write $\beta_1 = |\EG|-|\VG|+1$ for the first Betti number of $\DG$.
We compute
$$
\begin{aligned}
  \Bigl|\frac{\ker(\beta^G)}{\im(\alpha^G)}\Bigr|
  &\>=\>|\coker(\MBL)_{\tors}|
  \>\stackrel{\ref{BLsym}}{=}\>\frac{\prod_{v\in\VG} r_v}{m'\cdot N^{|\VG|-1}}\cdot|\coker(\MSym)_{\tors}|\\
  &\>\stackrel{\ref{TNC}}{=}\>
  \frac{\prod_{v\in\VG} r_v}{m'\cdot N^{|\VG|-1}}\cdot t(\cT_1)
  \>\stackrel{\ref{tamlaws}(d)}{=}\>
  \frac{\prod_{v\in\VG} r_v}{m'\cdot N^{|\VG|-1}}\cdot
  \prod_{e=\e vw\in\EG}\frac{Nr_e}{r_vr_w}\cdot t(\cT_2)\\
  &\>\stackrel{\ref{tamlaws}(f,g)}{=}\>
  \frac{\prod_{v\in\VG} r_v}{m'\cdot N^{|\VG|-1}}\cdot
  \prod_{\e vw\in\EG}\frac{Nr_{\e vw}}{r_vr_w}\cdot
  \prod_{v\in\VG}\Bigl(\frac{r_v}{m'}\Bigr)^{\deg v - 2}\cdot \Bigl(\frac{m'^2}{N}\Bigr)^{\beta_1}\cdot t(\DG)\\
  &\>=\>
  m'\cdot\frac{\prod_{e\in\EG} r_e}{\prod_{v\in\VG} r_v}\cdot t(\DG),
\end{aligned}
$$
using $\prod_{\e vw}r_vr_w = \prod_v r_v^{\deg v}$, $|\EG|-\beta_1 = |\VG|-1$, and
$\sum_v(\deg v - 2) = 2|\EG|-2|\VG| = 2(\beta_1 - 1)$ in the final equality.
\end{proof}

\section{Cohomological part}
\label{scoh}

We now determine the cohomological term
$\ker\bigl(\psi\colon H^1(G,\im\alpha)\to H^1(G,\ker\beta)\bigr)$, and identify it
with $\Psi_J$ of Definition~\ref{defpsi}. We keep the setting and the notation of \S\ref{s:arith}.

\begin{theorem}\label{kerpsi}
Pick a representative $v_0\in V$ for every orbit $v\in\VG$, let $H_v=\Stab_G(v_0)$, and
for every $v'\in v$ pick $g_{v'\shortto v_0}\in G$ that takes $v'$ to $v_0$.

\begin{itemize}
\item[(1)] The group $H^1(G, \ker\beta)$ is canonically isomorphic to $\Z/m'\Z$.

\item[(2)] The group $H^1(G, \im\alpha)$ fits in an exact sequence
$$
  0 \to H^1(G, \im\alpha) \xrightarrow{\>\partial_\psi\>} \Hom(G, \Q/\Z) \xrightarrow{\>f\>}
    \bigoplus_{v \in \VG} \Hom(H_v, \Q/\Z), \qquad f(\chi)_v=m_v\cdot\chi|_{H_v}.
$$

\item[(3)]
Explicitly, let $\chi\in\ker f$, lift it to a function $\tilde\chi\colon G\to\Q$, and define for $\sigma\in G$
$$
  \xi = \bigl(\,m_v\,\tilde\chi(g_{v'\shortto v_0})\bigr)_{v'} \in\Q^V, \qquad
  \omega(\sigma) \>=\> \xi\!-\!\sigma(\xi)\!+\!\tilde\chi(\sigma)\,\mathbf{m}, \qquad
  c(\sigma) \>=\> \alpha(\omega(\sigma)).
$$
Then $\omega\in C^1(G,\Z^V)$ is a 1-cocycle, $c\in Z^1(G,\im\alpha)$, and $\partial_\psi([c])=\chi$.

\item[(4)]
With $\chi$, $\tilde\chi$ and $\xi$ as in \textup{(3)},
\begin{equation}\label{starstar}
  \psi(\chi) \equiv \sum_{v\in\VG} m_v r_v\,(\alpha(\xi))_{v_0}
  \equiv
  \sum_{v \in \VG} m_v r_v
  \sum_{w' \in V} m_{w'}\,
  (v_0\cdot w')\,
  \tilde\chi(g_{w'\shortto w_0})
  \pmod{m'}.
\end{equation}
\end{itemize}
\end{theorem}

\begin{lemma}\label{lem:shapiro}
Let $G$ act transitively on a finite set $X$, with stabiliser $H$. Then
$$
  H^1(G,\Z^X) = 0 \qquad\text{and}\qquad
  H^2(G,\Z^X) \>\cong\> H^\vee=\Hom(H,\Q/\Z).
$$
Under the inclusion $\Z\hookrightarrow\Z^{X}$ sending
$1\mapsto\sum_{x\in X}x$, the induced map
$H^2(G,\Z)\to H^2(G,\Z^{X})$ corresponds to restriction
$\Gvee\to H^\vee$.
\end{lemma}

\begin{proof}
Since $\Z^{X}\cong\Ind_{H}^G\Z$, Shapiro's lemma
(see e.g.\ \cite[1.6.3]{NSW}) gives
$H^q(G,\Z^{X})\cong H^q(H,\Z)$ for all $q$.
Then $H^1(H,\Z)=\Hom(H,\Z)=0$ and
$H^2(H,\Z)\cong H^\vee$ by the long exact sequence of
$0\to\Z\to\Q\to\Q/\Z\to 0$.
The last assertion follows from the naturality of Shapiro's isomorphism:
the inclusion $\Z\hookrightarrow\Ind_{H}^G\Z$ corresponds on
$H$-modules to the identity $\Z\to\Z$, composed with
restriction $G\to H$.
\end{proof}

\begin{proof}[Proof of Theorem~\ref{kerpsi}]
As a $G$-module, $\Z^V=\bigoplus_{v\in\VG}\Z[G/H_v]$.
Applying Lemma~\ref{lem:shapiro} to the summands, we find
$H^1(G,\Z^V)=0$ and $H^2(G,\Z^V)\cong\bigoplus_{v\in\VG}H_v^\vee$.

\smallskip\noindent (1)
Since $\gcd(m_v)=1$, the map $\beta$ is surjective and
$0\to\ker\beta\to\Z^V\xrightarrow{\beta}\Z\to 0$ is exact.
The long exact sequence in $G$-cohomology, together with $H^1(G,\Z^V)=0$, gives
$H^1(G,\ker\beta)\cong\Z/\im\beta^G$. Since $\beta^G(\Vlift{v})=r_v m_v$,
$\im\beta^G=m'\Z$, whence $H^1(G,\ker\beta)\cong\Z/m'\Z$.

\smallskip\noindent (2)
Recall that $\ker\alpha=\Z\mathbf m$, and $\mathbf m$ is $G$-invariant.
Since $H^1(G,\Z^V)=0$, the short exact sequence
$0\to\mathbf m\Z\to\Z^V\xrightarrow{\alpha}\im\alpha\to 0$ yields an exact sequence
$$
  0 \to H^1(G,\im\alpha)\xrightarrow{\>\partial_\psi\>} H^2(G,\Z)
   \xrightarrow{\>f\>} H^2(G,\Z^V).
$$
Identify $H^2(G,\Z)\cong\Gvee$ and $H^2(G,\Z^V)\cong\bigoplus_v H_v^\vee$.
The $v$-component of the map $\Z\mathbf{m}\to\Z^V$ sends
$\mathbf{m}\mapsto m_v\Vlift{v}$, factoring as
$\Z\xrightarrow{\times m_v}\Z\xrightarrow{1\mapsto\Vlift{v}}\Z^{v}$.
The second map induces restriction $\Gvee\to H_v^\vee$ on $H^2$ by \ref{lem:shapiro},
so the $v$-component of $f$ is $\chi\mapsto m_v\,\chi|_{H_v}$.
Hence $H^1(G,\im\alpha)\cong\ker f$.

\smallskip\noindent (3)
Recall that $G$ acts on $\Q^V$ by $(\sigma\cdot\xi)_{v'}=\xi_{\sigma^{-1}v'}$.
Let $w\in V$. Since $m_w\chi: G\to \Q/\Z$ is a homomorphism, trivial on $H_w$ for $\chi\in\ker f$, we have
\begin{equation}\label{sigmaprod}
  \text{
  If $\sigma_1,\dots,\sigma_r\in G$ satisfy $\sigma_1\cdots\sigma_r\in H_w$,
  then $m_w\sum_i\tilde\chi(\sigma_i)\in\Z$.
  }
\end{equation}
\textit{Integrality of $\omega(\sigma)$.}
Applying \eqref{sigmaprod} to
$g_{\sigma^{-1}w'\shortto w_0}^{-1}\cdot g_{\sigma^{-1}w'\shortto w_0} \in H_w$ and
$g_{w'\shortto w_0}\cdot\sigma\cdot g_{\sigma^{-1}w'\shortto w_0}^{-1} \in H_w$,
we get that for any $w'\in w\in V$,
$$
  \omega(\sigma)_{w'} =
    \xi_{w'}-\xi_{\sigma^{-1}w'}+m_w\tilde\chi(\sigma) \equiv
  m_w\bigl(\tilde\chi(g_{w'\shortto w_0})
           +\tilde\chi(\sigma)
           -\tilde\chi(g_{\sigma^{-1}w'\shortto w_0})\bigr) \equiv 0\pmod\Z.
$$

\textit{Cocycle $c\in Z^1(G,\im\alpha)$.}
Since $\omega$ is integral, $c = \alpha\circ\omega$ lands in $\im\alpha$. Now,
$G$ preserves intersection numbers, so $\alpha$ is $G$-equivariant on $\Q^V$:
$$
  (\alpha(\sigma\xi))_{v'}
  =\sum_{w'}(v'\cdot w')\,\xi_{\sigma^{-1}w'}
  =\sum_{w''}(v'\cdot\sigma w'')\,\xi_{w''}
  =\sum_{w''}(\sigma^{-1}v'\cdot w'')\,\xi_{w''}
  =(\sigma\,\alpha(\xi))_{v'}.
$$
Using this and $\alpha(\mathbf{m})=0$, we get
$$
  c(\sigma) = \alpha(\omega(\sigma)) = \alpha(\xi-\sigma\xi+\tilde\chi(\sigma)\mathbf{m})
            = \alpha(\xi)-\sigma(\alpha(\xi)),
$$
so $c$ is automatically a 1-cocycle (despite $\xi$ being rational-valued).

\textit{Image of $c$.}
The connecting homomorphism $\partial_\psi$ works as follows: given
$c\in H^1(G,\im\alpha)$, lift it to a cochain $\omega\colon G\to\Z^V$
with $\alpha(\omega(\sigma))=c(\sigma)$; then
$(\sigma,\tau)\mapsto\omega(\sigma)+\sigma\omega(\tau)-\omega(\sigma\tau)$
is a 2-cocycle with values in $\ker\alpha=\Z\mathbf{m}$, representing
$\partial_\psi([c])\in H^2(G,\Z\mathbf{m})\cong H^2(G,\Z)\cong\Gvee$.
Computing directly with $\omega(\sigma)=\xi-\sigma\xi+\tilde\chi(\sigma)\mathbf{m}$,
$$
  \omega(\sigma)+\sigma\omega(\tau)-\omega(\sigma\tau)
  = \bigl(\tilde\chi(\sigma)+\tilde\chi(\tau)-\tilde\chi(\sigma\tau)\bigr)\mathbf{m}
$$
since the $\xi$ terms cancel. By definition, $F(\sigma,\tau)=\tilde\chi(\sigma)+\tilde\chi(\tau)
-\tilde\chi(\sigma\tau)$ is the 2-cocycle representing $\chi$ under
$H^2(G,\Z)\cong\Gvee$, so $\partial_\psi([c])=\chi$.

\smallskip\noindent (4)
The isomorphism $H^1(G,\ker\beta)\cong\Z/m'\Z$ in~(1) comes from a connecting homomorphism,
and is given as follows: for a cocycle $c\colon G\to\ker\beta$,
choose $u\in\Z^V$ with $c(\sigma)=\sigma u-u$ (possible since $H^1(G,\Z^V)=0$);
then $\beta(u)$ is $G$-invariant
and its class in $\Z/m'\Z$ is the image of $[c]$; that is,
$\psi(\chi)\equiv\beta(u)\pmod{m'}$.

By (3), $c(\sigma)=\alpha(\xi)-\sigma(\alpha(\xi))$ is in $\im\alpha$.
In other words, $-\alpha(\xi)$ behaves exactly like $u$ is supposed to, except that it is not $\Z$-valued.
However, $a=u+\alpha(\xi)\in\Q^V$ is $G$-invariant, hence constant on each $G$-orbit:
$a_v=u_{v_0}+(\alpha(\xi))_{v_0}$ for $v\in\VG$.
Using $\beta\alpha=0$,
$$
  \beta(u) \>=\> \beta(a)-\beta(\alpha\xi)
  \>=\> \sum_{v\in\VG} m_v r_v\,a_v
  \>=\> \sum_{v\in\VG} m_v r_v\,u_{v_0}\>+\>\sum_{v\in\VG} m_v r_v\,(\alpha\xi)_{v_0}.
$$
The first sum lies in $m'\Z$ since $m'\mid m_v r_v$ for every $v\in\VG$, so
$$
  \psi(\chi) \>\equiv\> \sum_{v\in\VG} m_v r_v\,(\alpha\xi)_{v_0}\pmod{m'},
$$
which is the first equality in \eqref{starstar}. The second follows by plugging in
$(\alpha\xi)_{v_0}=\sum_{w'\in V}(v_0\cdot w')\,\xi_{w'}$ and
$\xi_{w'}=m_{w'}\,\tilde\chi(g_{w'\shortto w_0})$.
\end{proof}

\begin{lemma}\label{charsum2}
Let $v\in\VG$, $\chi\in\Hom(G,\Q/\Z)$
with $\order(\chi|_{H_v})=n_v$, and $g\in G$. Then
$$
  \sum_{h\in H_v}\chi(hg) \>=\> \frac{|G|\,(n_v-1)}{2\,r_v\,n_v} + \frac{|G|}{r_v}\cdot\chi(g)
  \quad\in \Q/\Z.
$$
\end{lemma}
\begin{proof}
Let $K_v=\ker(\chi|_{H_v})$. The values of $\chi(hg)=\chi(h)+\chi(g)$ as $h$ ranges over $H_v$
take each value $\tfrac{j}{n_v}+\chi(g)$ exactly $|K_v|=\tfrac{|G|}{r_vn_v}$ times, so
$$
  \sum_{h\in H_v}\chi(hg)
  \>=\> \frac{|G|}{r_vn_v}\sum_{j=0}^{n_v-1}\Bigl(\frac{j}{n_v}+\chi(g)\Bigr)
  \>=\> \frac{|G|}{r_vn_v}\cdot\frac{n_v-1}{2} \>+\> \frac{|G|}{r_v}\cdot\chi(g).
$$
\end{proof}

\begin{theorem}\label{psifor}
In the setting of Theorem~\ref{kerpsi}, assume $G$ acts without
inversion on $\D$.
Let $\chi \in \im(\partial_\psi)\subset\Hom(G, \Q/\Z)$.
For $v \in \VG$ write $n_v=\order(\chi|_{H_v})$.
Then
\begin{enumerate}
\item
We have
$$
  \psi(\chi) \>\equiv\> \frac{m'}{2} \sum_{v \in \VG}
  \frac{r_v m_v}{m'} \cdot \frac{m_v}{n_v} \cdot
  (v_0 \cdot v_0) \cdot (n_v - 1)
  \quad\pmod{m'}.
$$

\item
All factors $r_v m_v/m'$, $m_v/n_v$, $v_0\cdot v_0$
and $n_v - 1$ are integers.
In particular, $2\psi(\chi) \equiv 0 \pmod{m'}$,
so the image of $\psi : H^1(G, \im\alpha) \to \Z/m'\Z$
is killed by~$2$.
\item
In the notation of Definition~\ref{defpsi},
$$
    \ker\bigl(\psi\colon H^1(G,\im\alpha)\to H^1(G,\ker\beta)\bigr) \>\iso\> \Psi_J.
$$
\end{enumerate}
\end{theorem}

\begin{proof}
(1, 2)
We start from the second expression in \eqref{starstar} (Theorem \ref{kerpsi}):
$$
  \psi(\chi) \equiv
  \sum_{v\in\VG} m_v r_v \sum_{w'\in V} m_{w'}\,(v_0\cdot w')\,\tilde\chi(g_{w'\shortto w_0})
  \pmod{m'},
$$
for any lift $\tilde\chi : G \to \Q$.
Choosing $g_{v_0\shortto v_0}=1$ and $\tilde\chi(1)=0$, the diagonal term $w'=v_0$
in the inner sum vanishes. Since $v_0\cdot w'$ counts geometric edges
from $v_0$ to $w'$, the remaining inner sum reindexes as a sum over geometric edges
at $v_0$, each edge $(v_0,w')$ contributing $m_{w'}\,\tilde\chi(g_{w'\shortto w_0})$.

Now we group by edge orbits. Fix an orbit $e=\e vw\in\EG$.
Since $G$ is transitive on~$v$, there is a geometric edge $(v_0,w')\in e$,
though not necessarily with $w'=w_0$.
Since $G$ acts without inversion, the edges of $e$ incident to $v_0$
are in one-to-one correspondence with $[h]\in H_v/H_e$ via $[h]\mapsto(v_0,h^{-1}w')$.

The contribution of $e$ to the $v_0$ term of \eqref{starstar} is
$$
  m_v r_v \cdot m_w \sum_{[h]\in H_v/H_e} \tilde\chi(g_{h^{-1}w'\shortto w_0})\mod m'\Z.
$$
By \eqref{sigmaprod}, the values $m_w\chi(g)\in\Q/\Z$ are constant on left and right cosets of $H_w$,
and therefore so are $m_vr_vm_w\tilde\chi(g)$ mod $m'\Z$. As $H_e\subseteq H_w$,
we can replace $\sum_{[h]\in H_v/H_e}$ by $\frac{1}{|H_e|}\sum_{h\in H_v}$
and choose $g_{h^{-1}w'\shortto w_0}=h\cdot g_{w'\shortto w_0}$ in the above sum,
and it becomes
$$
  \frac{m_vm_wr_v}{|H_e|}\sum_{h\in H_v}\tilde\chi(h\cdot g_{w'\shortto w_0})\mod m'\Z.
$$
Applying Lemma~\ref{charsum2} with $g=g_{w'\shortto w_0}$, this equals
$$
  \frac{m_vm_wr_v}{|H_e|}
  \Bigl(
    \frac{|G|\,(n_v-1)}{2\,r_v\,n_v} + \frac{|G|}{r_v}\cdot\tilde\chi(g_{w'\shortto w_0})
  \Bigr)
  =
  \frac{m_vm_w|G|}{|H_e|}
  \Bigl(
    \frac{n_v-1}{2\,n_v} + \tilde\chi(g_{w'\shortto w_0})
  \Bigr)\mod m'\Z.
$$
In the same way, the contribution of $e$ to the $w_0$ term is
$$
  \frac{m_vm_w|G|}{|H_e|}
  \Bigl(
    \frac{n_w-1}{2\,n_w} + \tilde\chi(g_{v'\shortto v_0})
  \Bigr).
$$
The point now is that the two `$\tilde\chi$'s cancel.
Since $(v_0,w')$ and $(v',w_0)$ are both in $e$, there is $\sigma\in G$
with $\sigma(v_0)=v'$ and $\sigma(w')=w_0$. So $g_{w'\shortto w_0}$ and $\sigma$
differ by left-multiplication by an element of $H_w$, and similarly
$g_{v'\shortto v_0}$ and $\sigma^{-1}$ differ by left-multiplication by an element of $H_v$.
By \eqref{sigmaprod}, $\chi(g_{w'\shortto w_0})+\chi(g_{v'\shortto v_0})\equiv\chi(\sigma)+\chi(\sigma^{-1})=0$
modulo terms that vanish mod $m'\Z$ after multiplication by $m_vm_w|G|/|H_e|$.
Adding the $v$- and $w$-contributions, we get
$$
  \psi(\chi) \equiv
  \sum_{e=\e vw \in \EG}
  \frac{m_vm_w|G|}{|H_e|}
  \Bigl(
    \frac{n_v-1}{2\,n_v} + \frac{n_w-1}{2\,n_w}
  \Bigr)\mod m'\Z.
$$
Rewriting this back as a sum over vertices,
$$
  \psi(\chi) \>\equiv\>
  \sum_{v\in\VG} \frac{|G|\,m_v\,(n_v-1)}{2\,n_v}
  \sum_{e\ni v} \frac{m_{w(e)}}{|H_e|}
  \pmod{m'},
$$
where $w(e)$ denotes the orbit at the far end of $e$ from $v$.
The inner sum is evaluated using the fibre relation $\mathbf{m}\cdot v_0=0$:
$$
  \sum_{e\ni v}\frac{m_{w(e)}}{|H_e|}
  = \frac{1}{|H_v|}\sum_{e\ni v} m_{w(e)}\,[H_v:H_e]
  = \frac{1}{|H_v|}\sum_{c\sim v_0} m_c
  = \frac{-m_v\,(v_0\cdot v_0)}{|H_v|}
  = \frac{-m_v r_v\,(v_0\cdot v_0)}{|G|},
$$
where the second equality counts geometric edges at $v_0$ weighted by
far-endpoint multiplicity, and the third uses $\sum_c m_c\,(c\cdot v_0)=0$.
Substituting,
$$
  \psi(\chi) \equiv
  -\frac{1}{2}\sum_{v\in\VG}
  \frac{m_v^2\,r_v\,(v_0\cdot v_0)\,(n_v\!-\!1)}{n_v}
  \equiv
  -\frac{m'}{2}\sum_{v\in\VG}
  \frac{r_v m_v}{m'}\cdot\frac{m_v}{n_v}\cdot(v_0\cdot v_0)\cdot(n_v\!-\!1)
  \pmod{m'}.
$$
Finally, by Theorem~\ref{kerpsi}(2), as $\chi \in \im\partial_\psi=\ker f$ we have
$m_v/n_v\in\Z$, and the other factors in the sum are integers by definition.
Therefore the whole expression is 2-torsion in $\Z/m'\Z$, and we can drop
the overall sign.

(3) Recall that $m_v/n_v=[\chi]_v$ in the notation of Definition~\ref{defpsi},
and the condition that this is an integer for all $v\in \bar V$ is equivalent to $\chi\in\im\partial_\psi$.
Denote
$$
  A_v=r_v m_v/m', \qquad
  B_v=m_v/n_v=[\chi]_v, \qquad
  C_v=v_0\cdot v_0, \qquad
  D_v=n_v-1.
$$
By (1) and (2) of the theorem,
$$
  \psi(\chi)=0 \qquad \Longleftrightarrow \qquad \sum_{v\in\VG} A_v B_v C_v D_v\in 2\Z,
$$
so we just need to explain why
$$
  A_v,B_v,C_v,D_v\text{ are all odd}
  \quad\Longleftrightarrow\quad
  v\text{ is extremal and }B_v\text{ is odd.}
$$
Now, $C_v$ odd is exactly the first condition for extremality in \ref{defpsi}, and $A_v$ odd
is the third. Finally, when $B_v=m_v/n_v$ is odd,
the condition that $D_v=n_v-1$ is odd, equivalently $n_v$ is even, is equivalent to $m_v$ being even,
the second condition for extremality.
\end{proof}

Recall from \S\ref{shistory} that together with Theorem \ref{tarith}
this completes the proof of Theorem \ref{etnc}.

\section{Appendix A. Linear algebra}

\begin{theorem}
\label{pairings}
Let $W$ be a finite-dimensional $\Q$-vector space equipped with a
nondegenerate symmetric $\Q$-valued bilinear form $\langle\cdot,\cdot\rangle$.

\begin{itemize}
\item[(1)] (Determinant)
If $v_1,\dots,v_r$ is a $\Z$-basis of a lattice $L \subset W$,
write $Q$ for the Gram matrix of $\langle\cdot,\cdot\rangle$ with respect to this basis, so that
$$
  Q=\bigl(\langle v_i,v_j\rangle\bigr)_{i,j}, \qquad
  \Bigl\langle \sum_{i=1}^r x_i v_i , \sum_{i=1}^r y_i v_i \Bigr\rangle = x^T Q y.
$$
The determinant $\det Q$ is independent of the basis choice, and is denoted $\det(\lara|_L)$.%
\footnote{The determinant may be zero if $L$ is not of full rank.}

\item[(2)] (Pullback formula)
Let $Q$ be the Gram matrix of $\langle\cdot,\cdot\rangle$ in some basis of $W$.
Let $A:\Z^r \to W$ be an injective homomorphism, viewed as a matrix
with respect to the standard basis of $\Z^r$ and the chosen basis of $W$. Then
$$
  \det(\lara|_{A(\Z^r)}) = \det(A^T Q A).
$$

\item[(3)] (Index--determinant relation)
If $L' \subset L$ is a sublattice of finite index, then
$$
  \det(\lara|_{L'}) = [L:L']^2 \det(\lara|_L).
$$

\item[(4)] (Orthogonal decomposition formula)
Suppose there is an orthogonal decomposition over $\Q$
$$
  W = U \oplus U'.
$$
Let $L \subset W, L_U \subset U, L_{U'} \subset U'$ be lattices
such that $L_U \oplus L_{U'} \subseteq L$, and is of finite index. Then
$$
  \det(\lara|_L)\cdot[L:L_U\oplus L_{U'}]^2
  \>=\> \det(\lara|_{L_U})\cdot\det(\lara|_{L_{U'}}).
$$
\end{itemize}
\end{theorem}

\begin{proof}
This is classical, see. e.g. \cite[Ch. 1]{Cassels}.
\end{proof}

The following lemmas extend Lorenzini's \cite[Lemma 9.6.5]{BLR} to non-symmetric matrices:

\begin{lemma}\label{adjug}
Let $A\in M_n(\Z )$ have rank $n-1$.
Write $A^*$ for the classical adjugate (or adjoint) matrix of $A$,
that is $A^* = (\det A_{ji})$ where $A_{ji}$ is the $(j,i)$ cofactor, so that $A A^* = \det(A) I$.
If $\mathbf{l}^T A=0$, $A\mathbf{m}=0$ with
$\mathbf{l}, \mathbf{m}\in \Z ^n\setminus\{0\}$, then
$$
  A^* \>=\> \pm\,\frac{|(\coker A)_{\tors}|}{\gcd(\mathbf{m})\gcd(\mathbf{l})}\cdot \mathbf{m}\,\mathbf{l}^T
     \qquad\text{and}\qquad
  \gcd(A^*)  \>=\> |(\coker A)_{\tors}|.
$$
\end{lemma}

\begin{proof}
Write $A$ in Smith normal form, so $A=UDV$ with
$D=\mathrm{diag}(d_1,\dots,d_{n-1},0)$ and $U,V\in \GL_n(\Z)$.  Using
$(XY)^*=Y^*\,X^*$ and
$U^*=\det(U)U^{-1}=\pm U^{-1}$ and similarly for $V$,
$$
  A^* \>=\> V^*\,D^*\,U^*.
$$
Since $D$ has exactly one zero diagonal entry, $D^*$ has exactly
one nonzero entry, namely $\prod_{i=1}^{n-1}d_i=|(\coker A)_{\tors}|$ in position $(n,n)$,
so $D^*=t(A)\,e_n e_n^T$.  Therefore
$$
  A^* \>=\> t(A)\cdot
  \bigl(V^*\,e_n\bigr)\bigl(e_n^T\,U^*\bigr)
  \>=\> \pm\,t(A)\cdot
  (V^{-1}e_n)(e_n^T U^{-1}).
$$
The vector $V^{-1}e_n$ is the last column of $V^{-1}$, which spans
$\ker(A)=\ker(DV)$ over $\Q $; since it is a column of the
unimodular matrix $V^{-1}$ it is primitive, hence equals $\pm\mathbf{m}$.
Likewise $e_n^T U^{-1}$ is the last row of $U^{-1}$, which spans the left
kernel of $A$ and is primitive, hence equals $\pm\mathbf{l}^T$.

The last claim follows from
$$
  \gcd_{i,j} \frac{m_i l_j}{\gcd(\mathbf{m})\,\gcd(\mathbf{l})} =
    \gcd_i \frac{m_i}{\gcd(\mathbf{m})}\>\gcd_j \frac{l_j}{\gcd(\mathbf{l})} = 1.
\rlap{\qquad\qquad\qquad\qquad\qedhere}
$$
\end{proof}

\begin{lemma}\label{DMfor}
Let $M\in M_n(\Z)$ have rank $n-1$, and $D=\diag(r_1,\dots,r_n)\in M_n(\Z)$
diagonal with $r_i>0$.
Let $\mathbf{l}=(l_1,\dots,l_n)$ be any non-zero vector such that $\mathbf{l}^T M=0$.
Write (the quantities $d_i, d$ below are local to this lemma, unrelated to the
Bosch--Liu $d_v$ of \S\ref{sscoh})
$$
  R=\prod_{i=1}^n r_i, \qquad
  d_i \>=\> \frac{R}{r_i}\cdot\dfrac{l_i}{\gcd(\mathbf{l})} \qquad
  d \>=\> R/\gcd_i(d_i).
$$
Then
\begin{itemize}
\item[(1)]
The matrix $DM$ has rank $n-1$.
\item[(2)]
There is a short exact sequence
$$
  0 \to \coker(M) \xrightarrow{\bar{D}} \coker(DM) \xrightarrow{\gamma} \bigoplus_{i=1}^n \Z/r_i\Z \to 0,
$$
where $\bar{D}$ is induced by $D$, and $\gamma$ is the natural projection
$\Z^n/D(\im M)\to\Z^n/D(\Z^n)$.
\item[(3)]
Modulo torsion, the map $\bar D$ in (2) is multiplication by $d$: $\Z\to\Z$.
\item[(4)]
There is an exact sequence
$$
  0 \to \coker(M)_{\tors} \xrightarrow{\bar{D}} \coker(DM)_{\tors}
    \xrightarrow{\gamma}
    \bigoplus_{i=1}^n \Z/r_i\Z
    \xrightarrow{\pi} \Z/d\Z \to 0,
$$
where $\pi$ sends $e_i\mapsto d_i/\gcd_j(d_j)$.
\end{itemize}
\end{lemma}

\begin{proof}
Replacing $\mathbf{l}$ by $\mathbf{l}/\gcd(\mathbf{l})$ if necessary,
we may assume $\mathbf{l}$ is primitive.

\textbf{(1)} Since $D$ is invertible over $\Q$, we get $\rk(DM)=\rk(M)=n-1$.

\textbf{(2)} Since $D$ is injective and $\im(DM)=D(\im M)\subset D\Z^n$, the sequence
$$
  0\to \Z^n/\im M\xrightarrow{[v]\,\mapsto\,[Dv]} \Z^n/\im(DM)\to \Z^n/D\Z^n\to 0
$$
is exact, and $\Z^n/D\Z^n\cong\bigoplus_i\Z/r_i\Z$.

\textbf{(3),(4)}
The maps $[v]\mapsto{\mathbf{l}}^T v$ and $[v]\mapsto\hat{\mathbf{l}}^T v$
are surjections $\coker(M)\to\Z$ and $\coker(DM)\to\Z$ with kernels equal to
the respective torsion subgroups, where $\hat{\mathbf{l}}$ is the primitive left
null vector of $DM$: the left kernel of $DM$ over $\Q$ is spanned by
$D^{-T}{\mathbf{l}}$, whose $i$-th component is ${l}_i/r_i$,
so $\hat{l}_i = d_i/\gcd_j(d_j)$.
Indeed,
$$
  \hat{l}_i r_i \>=\> \frac{d_i}{g}\cdot r_i \>=\> \frac{R}{r_i}\cdot{l}_i\cdot\frac{r_i}{g}
  \>=\> \frac{R}{g}\,{l}_i \>=\> d\,{l}_i, \qquad g=\gcd_i d_i.
$$
So $\hat{\mathbf{l}}^T D = d\,{\mathbf{l}}^T$, and $\bar{D}$ acts
as multiplication by $d$ on the free summands, proving~(3).
For~(4), define $\pi\colon\bigoplus_i\Z/r_i\Z\to\Z/d\Z$ by
$\pi(\mathbf{a})=\hat{\mathbf{l}}^T\mathbf{a}\bmod d$; this is well-defined
(if $a_i'-a_i=r_i w_i$ then
$\hat{\mathbf{l}}^T(\mathbf{a}'-\mathbf{a})=d\,{\mathbf{l}}^T\mathbf{w}\equiv 0$)
and surjective (since $\gcd_i\hat{l}_i=1$).
The diagram
$$
  \begin{array}{ccccccccc}
    0 & \to & \coker(M) & \xrightarrow{\bar{D}} & \coker(DM) & \xrightarrow{\gamma}
      & \bigoplus_i\Z/r_i\Z & \to & 0 \\[2pt]
      && \downarrow\rlap{$\scriptstyle{\,{\mathbf{l}}^T}$}
      && \downarrow\rlap{$\scriptstyle{\,\hat{\mathbf{l}}^T}$}
      && \downarrow\rlap{$\scriptstyle{\,\pi}$} \\
    0 & \to & \Z & \xrightarrow{\>\times d\>} & \Z & \to & \Z/d\Z & \to & 0
  \end{array}
$$
commutes with exact rows and surjective vertical maps, and the kernels give (4).
\end{proof}

\section{Appendix B. Tamagawa numbers of weighted multigraphs}
\label{stamgraph}

We formulate basic properties of Tamagawa numbers of weighted graphs, taken from
electrical network theory. This is similar to \cite{CR}, except with vertex weights in addition to edge weights.
One observation, perhaps of independent interest, is that the
standard reduction laws (tail removal, serial, parallel, Y-$\Delta$) all follow formally from
the star-mesh transform \cite{Kennelly}.

A \textbf{weighted graph} is a finite connected multigraph
$$
  \cT = \bigl(V(\cT),E(\cT),(m_v)_v,(s_e)_e\bigr),
$$
in which vertices $v$ carry \textbf{multiplicity} $m_v\in\R_{>0}$ and
edges $e$ carry \textbf{length} $s_e\in\R_{>0}$.
It is naturally a metric space, every edge $e$ being isometric to
$[0,s_e]\subset{\mathbb R}$.
Loops are allowed, and contribute $2$ to the degree of their vertex.
Define the length pairing $\langle e,e'\rangle=\delta_{ee'}s_e$,
and the \textbf{Tamagawa number}
\begin{equation}\label{eqtDwt}
  t(\cT) = \prod_{v\in V} m_v^{\deg v - 2} \>\cdot\>
  \det\bigl(\langle\cdot,\cdot\rangle\big|_{H_1(\cT,\Z)}\bigr).
\end{equation}

\begin{theorem}
\label{tamwt}
The Tamagawa number $t$ of weighted graphs satisfies the following:
\begin{itemize}

\item[(a)] (Single vertex)
If $\cT$ has a single vertex $v$ and no loops, then
$$
  t(\cT) \>=\> m_v^{-2}.
$$

\item[(b)] (Loops)
If $\cT'$ is obtained from $\cT$ by adding an extra loop $e$ at a vertex $v$, then
$$
  t(\cT') \>=\> s_e\, m_v^2\, t(\cT).
$$

\item[(c)] (Star-mesh transform.)
Let $w$ be a vertex of $\cT$ of degree $n\ge 1$
with incident edges $e_1,\ldots,e_n$ of lengths $s_1,\ldots,s_n$ going to $n$
distinct neighbours $v_1,\ldots,v_n$. Set $\kappa=\sum_{k=1}^n s_k^{-1}$.
Let $\cT'$ be obtained from $\cT$ by removing $w$ and $e_1,\ldots,e_n$ and
adding an edge of length $s_is_j\kappa$ between each pair $v_i,v_j$. Then
$$
  t(\cT') \>=\> m_0^{-(n-2)}\cdot\prod_{k=1}^n (m_{v_k} s_k)^{n-2}
  \cdot \kappa^{\binom{n}{2}-1}\cdot t(\cT).
$$

\item[(d)]
The Tamagawa number is the unique function on weighted graphs satisfying
\textup{(a)}-\textup{(c)}.
\end{itemize}
\end{theorem}

\begin{proof}
\textit{(a)} Since $\cT$ has no edges, $H_1=0$ and $t(\cT) = m_v^{0-2}$.

\textit{(b)} The loop $e$ gives an orthogonal decomposition
$$
  H_1(\cT,\Z) \>=\> \Z e \oplus H_1(\cT',\Z).
$$
Taking determinants of $\lara$ and comparing $\deg v$ in $\cT$ and $\cT'$ we get
$t(\cT) = s_e\,m_v^2\,t(\cT')$.

\noindent
\textit{(c)}
\textit{Vertex factors.}
Write $c_k=s_k^{-1}$ for the `edge conductances'.
Removing $w$ of degree $n$ removes the factor $m_w^{n-2}$ from
$\prod_v m_v^{\deg(v)-2}$. In $\cT'$, each neighbour $v_k$ loses one
edge to $w$ and gains $n-1$ mesh edges, so
$\deg_{\cT'}(v_k) = \deg_\cT(v_k) + (n-2)$, contributing an extra
factor $m_{v_k}^{n-2}$. All other vertex degrees are unchanged, so
\begin{equation}\label{eq:starmesh-vert}
  \frac{\prod_v m_v^{\deg_{\cT'}(v)-2}}{\prod_v m_v^{\deg_\cT(v)-2}}
  \>=\> \frac{\prod_{k=1}^n m_{v_k}^{n-2}}{m_w^{n-2}}.
\end{equation}

\textit{Homology determinant.}
It remains to compute
$\det\bigl(\langle\cdot,\cdot\rangle\big|_{H_1(\cT,\Z)}\bigr)/
\det\bigl(\langle\cdot,\cdot\rangle\big|_{H_1(\cT',\Z)}\bigr)$.

Write $c_e = s_e^{-1}$ for the `edge conductances' and let
$\tau(\cT) = \sum_T \prod_{e\in T} c_e$ be the conductance-weighted
spanning-tree count, the sum being over all spanning trees $T$ of $\cT$. The
matrix-tree theorem for the homology pairing \textup{\cite[Cor.~1.6]{ABKS}} gives
$$
  \det\bigl(\langle\cdot,\cdot\rangle\big|_{H_1(\cT,\Z)}\bigr)
  \>=\> \sum_{T} \prod_{e\notin T} s_e
  \>=\> \prod_{e\in E(\cT)}s_e \cdot \tau(\cT),
$$
the second equality from
$\prod_{e\notin T} s_e = \bigl(\prod_e s_e\bigr)\prod_{e\in T}c_e$; similarly for $\cT'$.

The edges of $\cT'$ are those of $E(\cT)\setminus\{e_1,\ldots,e_n\}$ together
with the $\binom{n}{2}$ mesh edges of lengths $s_is_j\kappa$; each $s_k$
appears in $n-1$ of them, so
\begin{equation}\label{dethomT}
  \frac{\det\bigl(\langle\cdot,\cdot\rangle\big|_{H_1(\cT',\Z)}\bigr)}
       {\det\bigl(\langle\cdot,\cdot\rangle\big|_{H_1(\cT,\Z)}\bigr)}
  \>=\>
  \frac{\prod_{e\in E(\cT')} s_e}{\prod_{e\in E(\cT)} s_e}
    \cdot    \frac{\tau(\cT')}{\tau(\cT)}
  \>=\> \prod_{k=1}^n s_k^{n-2}\cdot \kappa^{\binom{n}{2}}
    \cdot    \frac{\tau(\cT')}{\tau(\cT)}
\end{equation}
To conclude (c) from
\eqref{eq:starmesh-vert} and \eqref{dethomT},
it remains to show that $\tau(\cT) = \kappa\,\tau(\cT')$.

The weighted Laplacian of $\cT$ is the $V(\cT)\times V(\cT)$ matrix $L(\cT)$ with
$$
  L(\cT)_{vw} \>=\>
  \begin{cases}
    \sum_{e\ni v} c_e & \text{if } v=w, \\
    -\sum_{e\colon v-w} c_e & \text{if } v\ne w,
  \end{cases}
$$
the sums running over the edges incident to $v$, respectively joining $v$ and $w$.
For a vertex $v_\ell$ let the reduced Laplacian $L(\cT)^{(v_\ell)}$ be the principal
submatrix of $L(\cT)$ on $V(\cT)\setminus\{v_\ell\}$ (i.e. removing one row and column).
By the matrix-tree theorem \textup{\cite[Cor.~5.6]{ABKS}},
\begin{equation}\label{eq:starmesh-kirchhoff}
  \tau(\cT) \>=\> \det L(\cT)^{(v_\ell)}
\end{equation}
for any choice of $v_\ell\in V(\cT)$. Pick such a $v_\ell\ne w$,
so that $\tau(\cT) = \kappa\,\tau(\cT')$ is equivalent to
\begin{equation}\label{eq:starmesh-redlap}
  \det L(\cT)^{(v_\ell)} \>=\> \kappa\cdot\det L(\cT')^{(v_\ell)}.
\end{equation}
Order the vertices as
$w, v_1, \ldots, v_n, \ldots$. Since $w$ has degree $n$
with incident edges $e_1,\ldots,e_n$, its diagonal entry is
$L(\cT)_{ww} = \sum_{v_k} c_k = \kappa$, and $L(\cT)$ has block form
$$
  L(\cT) \>=\>
  \begin{pmatrix} \kappa & -\mathbf{c}^T \\ -\mathbf{c} & L_{**} \end{pmatrix},
  \qquad \mathbf{c} = (c_1,\ldots,c_n,0,\ldots,0)^T,
$$
with $L_{**}$ the principal submatrix on $V(\cT)\setminus\{w\}$.
Then
\begin{equation}\label{eq:starmesh-schur}
  L_{**} - \mathbf{c}\,\mathbf{c}^T/\kappa \>=\> L(\cT'),
\end{equation}
in other words the Schur complement of the $w$-block is exactly $L(\cT')$.
Indeed, the new mesh edge between $v_i$ and $v_j$ has conductance
$1/(s_is_j\kappa) = c_ic_j/\kappa$, contributing $-c_ic_j/\kappa$ to the
$(v_i,v_j)$-entry, which matches $-\mathbf{c}\,\mathbf{c}^T/\kappa$ off the
diagonal; on the diagonal,
$$
  L(\cT')_{v_iv_i} - (L_{**})_{v_iv_i}
  \>=\> \underbrace{\sum_{j\ne i} c_ic_j/\kappa}_{\text{new mesh edges}}
        \>-\> \underbrace{c_i}_{\text{deleted }e_i}
  \>=\> \frac{c_i(\kappa-c_i)}{\kappa} - c_i
  \>=\> -\frac{c_i^2}{\kappa}.
$$
Now delete the $v_\ell$-row and $v_\ell$-column from~\eqref{eq:starmesh-schur}.
Deleting a non-$w$ index commutes with forming the Schur
complement of the $w$-block: the $w$-block of $L(\cT)^{(v_\ell)}$ is still
the scalar $\kappa$, and the Schur complement of that block is
$L(\cT')^{(v_\ell)}$.
Applying the Schur determinant identity
$\det\smallmatrix{a}{-B^T}{-B}{D} = a\cdot\det(D-B a^{-1}B^T)$
to the $w$-block yields~\eqref{eq:starmesh-redlap}.

\textit{(d)} Any connected weighted graph reduces to a single vertex with no loops by finitely many
applications of star-mesh (c), loop removal (b) and merging parallel edges. The latter is a special case of
(c) as the proof of \ref{tamlaws}(b,c) shows below, so $t$ is unique.
\end{proof}

\begin{corollary}[Tamagawa number laws for weighted graphs]
\label{tamlaws}
Standard transformations of a weighted graph $\cT$ into $\cT'$
change the Tamagawa number as follows:
\smallskip
\begin{center}
\renewcommand{\arraystretch}{1.4}
\begin{tabular}{l@{\qquad}c@{\qquad}c@{\qquad}c}
Law & $\cT$ & $\cT'$ & $t(\cT)/t(\cT')$ \\
\hline
(a) Leaf removal &
  \showT{\vertT{Tv0}{w}{0,0} \vertT{Tv1}{v}{2.6,0} \draw (Tv0) -- node[above, midway, elbl] {$s$} (Tv1); \Rstubs{Tv1}} &
  \showT{\vertT{Tv1}{v}{0,0} \Rstubs{Tv1}} &
  $m_v/m_w$ \cr
(b) Serial law &
  \showT{\vertT{Tu}{u}{0,0}\vertT{Tv}{w}{1.8,0}\vertT{Tw}{v}{3.6,0}\draw (Tu) -- node[above, midway, elbl] {$s$} (Tv);  \draw (Tv) -- node[above, midway, elbl] {$s'$} (Tw); \Lstubs{Tu}  \Rstubs{Tw}} &
  \showT{\vertT{Tu}{u}{0,0}\vertT{Tw}{v}{2.6,0}\draw (Tu) -- node[above, midway, elbl] {$s+s'$} (Tw); \Lstubs{Tu} \Rstubs{Tw}} &
  $1$ \cr
(c) Parallel law &
  \showT{\vertT{Tv}{u}{0,0}  \vertT{Tw}{v}{2.6,0}  \draw (Tv) to[bend left=22] node[above, midway, elbl] {$s$} (Tw);  \draw (Tv) to[bend right=22] node[above, midway, elbl] {$s'$} (Tw);  \Lstubs{Tv}  \Rstubs{Tw} } &
  \showT{\vertT{Tv}{u}{0,0} \vertT{Tw}{v}{2.6,0} \draw (Tv) to node[above, midway, elbl, scale=1.2] {$\frac{ss'}{s+s'}$} (Tw); \Lstubs{Tv} \Rstubs{Tw}} &
  $m_um_v(s\!+\!s')$  \cr
(d) \vtop{\hbox{$n$-parallel law}\vspace{2pt}\hbox{($n$ edges $u$--$v$)}} &
  \showT{\vertT{Tv}{u}{0,0}\vertT{Tw}{v}{2.6,0}\draw (Tv) to[bend left=40] node[above, midway, elbl] {$s$} (Tw);\draw (Tv) to[bend left=10] node[above, midway, elbl, inner sep=1.4pt] {$s$} (Tw);\node at (1.3,-0.22) {$\cdots$};\draw (Tv) to[bend right=40] node[below, midway, elbl] {$s$} (Tw);\Lstubs{Tv}\Rstubs{Tw}} &
  \showT{\vertT{Tv}{u}{0,0}\vertT{Tw}{v}{2.6,0}\draw (Tv) to node[above, midway, elbl, scale=1.3] {$\frac sn$} (Tw); \Lstubs{Tv} \Rstubs{Tw}} &
  $n(sm_um_v)^{n-1}$ \cr
(e) \vtop{\hbox{Y-$\Delta$ transform}\vspace{2pt}\hbox{\rm ($\kappa=\frac{1}{s_1}\!+\!\frac{1}{s_2}\!+\!\frac{1}{s_3}$)}} &
  \showTsh[10]{\vertTl{Tv1}{v_1}{0,0}{below} \vertTl{Tv2}{v_2}{1.8,1.1}{left} \vertTl{Tv3}{v_3}{3.6,0}{below} \vertTl{Tv0}{w}{1.8,0}{below} \draw (Tv1) -- node[below, midway, elbl] {$s_1$} (Tv0); \draw (Tv0) -- node[right, midway, elbl] {$s_2$} (Tv2); \draw (Tv0) -- node[below, midway, elbl] {$s_3$} (Tv3); \Lstubs{Tv1} \Ustubs{Tv2} \Rstubs{Tv3}} &
  \showTsh[10]{\vertTl{Tv1}{v_1}{0,0}{below} \vertTl{Tv2}{v_2}{1.3,1.1}{above right = -3pt and 2pt} \vertTl{Tv3}{v_3}{2.6,0}{below} \draw (Tv1) -- node[above left, pos=0.4, elbl, inner sep=1pt] {$s_1s_2\kappa$} (Tv2); \draw (Tv1) -- node[below, midway, elbl, inner sep=3pt] {$s_1s_3\kappa$} (Tv3); \draw (Tv2) -- node[above right, pos=0.65, elbl, inner sep=1pt] {$s_2s_3\kappa$} (Tv3); \Lstubs{Tv1} \Ustubs{Tv2} \Rstubs{Tv3}} &
  $\dfrac{m_w}{m_{v_1}m_{v_2}m_{v_3}\,s_1s_2s_3\,\kappa^2}$ \cr
(f) Vertex scaling &
  \multicolumn{2}{c}{replace each $m_v$ by $\lambda_v m_v$} &
  $\prod_v\lambda_v^{2-\deg v}$ \cr
(g) Edge scaling &
  \multicolumn{2}{c}{replace each $s_e$ by $\lambda\, s_e$} &
  $\lambda^{-\dim H_1(\cT,\Q)}$ \cr
\end{tabular}
\end{center}
In every case $w$ has exactly the edges shown, and it is removed in (a), (b) and (e); all other
vertices and all edges between them are unchanged.
\end{corollary}

\begin{proof}
\textit{(a)} This is a special case of \ref{tamwt}\,(c) for $n=1$: here $\kappa=1/s$ and
$$
  t(\cT')=m_0\cdot(m_1s)^{-1}\cdot s\cdot t(\cT)=(m_0/m_1)\,t(\cT).
$$
\textit{(b)} Apply \ref{tamwt}\,(c) with $n=2$;
all three exponents $-(n-2)$, $n-2$, $\binom{n}{2}-1$ vanish.

\noindent
\textit{(c)}  Choose $a, b > 0$ with $a+b=s'$, and set
$$
  s_1 = \frac{as}{s+s'}, \qquad s_2 = \frac{ab}{s+s'}, \qquad s_3 = \frac{sb}{s+s'},
  \qquad \kappa = \frac{1}{s_1}+\frac{1}{s_2}+\frac{1}{s_3} = \frac{(s+s')^2}{abs}.
$$
These satisfy
$$
  s_1s_2\kappa = a, \quad s_1s_3\kappa = s, \quad s_2s_3\kappa = b, \quad
  s_1s_2s_3\kappa^2 = s+s', \quad s_1+s_3 = \frac{ss'}{s+s'}.
$$
Let $\cT_Y$ be obtained from $\cT$ by removing $e, e'$, adding new vertices
$v_0, u$ of multiplicity $1$, and edges $v_0v, v_0u, v_0w$ of lengths
$s_1, s_2, s_3$:
$$
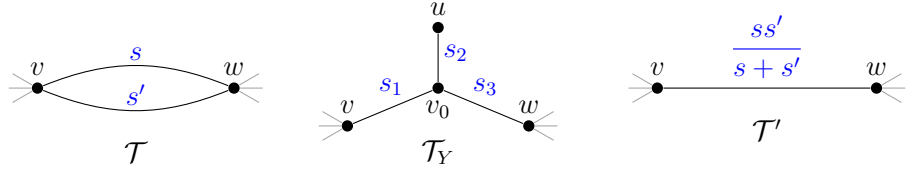
\showT[1.0]{
  \vertT{Tv}{v}{0,0}
  \vertT{Tw}{w}{2.6,0}
  \draw (Tv) to[bend left=22]  node[above, midway, elbl] {$s$}  (Tw);
  \draw (Tv) to[bend right=22] node[above, midway, elbl] {$s'$} (Tw);
  \Lstubs{Tv}
  \Rstubs{Tw}
  \node[glbl] at (1.3,-0.85) {$\cT$};

  \vertT{Yv}{v}{4.1,-0.5}
  \vertT{Yw}{w}{6.5,-0.5}
  \vertTl{Yv0}{v_0}{5.3,0}{below}
  \vertT{Yu}{u}{5.3,0.8}
  \draw (Yv)  -- node[above left,  pos=0.7, elbl] {$s_1$} (Yv0);
  \draw (Yv0) -- node[right,       pos=0.6, elbl] {$s_2$} (Yu);
  \draw (Yv0) -- node[above right, pos=0.3, elbl] {$s_3$} (Yw);
  \Lstubs{Yv}
  \Rstubs{Yw}
  \node[glbl] at (5.3,-0.85) {$\cT_Y$};

  \vertT{Pv}{v}{8.2,0}
  \vertT{Pw}{w}{11.1,0}
  \draw (Pv) -- node[above, midway, elbl, inner sep=4pt] {$\dfrac{ss'}{s+s'}$} (Pw);
  \Lstubs{Pv}
  \Rstubs{Pw}
  \node[glbl] at (9.65,-0.55) {$\cT'$};
}
$$

First, $t(\cT) \>=\> m_vm_w(s+s')\,t(\cT_Y)$. Indeed,
apply \ref{tamwt}(c) at $v_0$ in $\cT_Y$. The new mesh edges between
$(v,u), (v,w), (u,w)$ have lengths $a, s, b$, so together with the unchanged edges
of $\cT$ the result is $\cT$ with $e'$ subdivided at $u$. By the serial law (b),
this graph has Tamagawa number $t(\cT)$. The star-mesh factor
is $m_vm_w\,s_1s_2s_3\kappa^2 \>=\> m_vm_w(s+s')$.

Next, $t(\cT_Y) \>=\> t(\cT')$. Indeed,
$u$ is leaf (vertex of degree $1$) in $\cT_Y$, and (a) eliminates it
without changing $t$. After that, $v_0$ gets degree $2$ with neighbours $v, w$
at lengths $s_1, s_3$; the serial law (b) replaces this path by a single edge
of length $s_1\!+\!s_3 = ss'/(s+s')$, recovering~$\cT'$.

\noindent
\textit{(d)} Apply (c) to merge $e_1$ and $e_2$ into a single edge of
length $s/2$, giving a factor $2sm_vm_w$. Now, inductively, at step
$k$, merge the current edge of length $s/k$ with $e_{k+1}$ of length $s$,
giving length $s/(k+1)$ and factor $\frac{(k+1)s}{k}m_vm_w$. After $n-1$
steps the total factor is
$$
  \prod_{k=1}^{n-1} \frac{(k+1)s}{k}\,m_vm_w
  \>=\> (sm_vm_w)^{n-1}\prod_{k=1}^{n-1}\frac{k+1}{k}
  \>=\> n\,(sm_vm_w)^{n-1},
$$
and the surviving edge has length $s/n$, as required.

\noindent
\textit{(e)} This is \ref{tamwt}\,(c) for $n=3$.

\noindent
\textit{(f, g)} Clear from the definition of the Tamagawa number.
\end{proof}


\newpage

\begin{thebibliography}{99}

\bibitem{Alar}
D. C. Alar, J. Celaya, L. D. Garc\'\i a-Puente, M. Henson, A. K. Wheeler,
The sandpile group of a thick cycle graph,
Electron. J. Graph Theory Appl. 10 (2022), no.~2, 625--636.

\bibitem{Bacher}
R. Bacher, P. de la Harpe, T. Nagnibeda,
The lattice of integral flows and the lattice of integral cuts on a finite graph,
Bull. Soc. Math. France 125 (1997), no.~2, 167--198.

\bibitem{ABKS}
Y. An, M. Baker, G. Kuperberg, F. Shokrieh,
Canonical representatives for divisor classes on tropical curves
and the Matrix-Tree Theorem,
Forum Math. Sigma 2 (2014), e24.

\bibitem{BNi}
M. Baker, J. Nicaise,
Weight functions on Berkovich curves, Algebra and Number Theory 10:10 (2016), 2053--2079.

\bibitem{BNo}
M. Baker, S. Norine,
Riemann--Roch and Abel--Jacobi theory on a finite graph,
Adv. Math. 215 (2007), 766--788.

\bibitem{Betts}
A. Betts,
Tamagawa numbers of semistable abelian varieties over local fields,
Algebra and Number Theory 12 (2018), 1635--1671.

\bibitem{BL}
S. Bosch, Q. Liu,
Rational points of the component group of a N\'eron model,
Manuscripta Math. 98 (1999), 275--293.

\bibitem{BLR}
S. Bosch, W. L\"utkebohmert, M. Raynaud, N\'eron Models,
Ergebnisse der Mathematik und ihrer Grenzgebiete (3), Vol.~21, Springer-Verlag, Berlin, 1990.

\bibitem{Cassels}
J. W. S. Cassels, Rational Quadratic Forms, LMS Monographs 13, Academic Press, London, 1978.

\bibitem{CR}
T. Chinburg, R. Rumely, The capacity pairing, J. Reine Angew. Math. 434 (1993), 1--44.

\bibitem{CTS}
J.-L. Colliot-Th\'el\`ene, S. Saito,
Z\'ero-cycles sur les vari\'et\'es $p$-adiques et groupe de Brauer,
Internat. Math. Res. Notices 1996, no.~4, 151--160.

\bibitem{newton}
T. Dokchitser, Models of curves over discrete valuation rings, Duke Math. J. 170, no. 11 (2021), 2519--2574.

\bibitem{types}
T. Dokchitser, A classification of reduction types of curves, preprint, 2025, arXiv: 2512.09082.

\bibitem{HN}
L. H. Halle, J. Nicaise,
Motivic zeta functions of degenerating abelian varieties and the monodromy conjecture,
Adv. Math. 227 (2011), 610--653.

\bibitem{HNbase}
L. H. Halle, J. Nicaise,
N\'eron Models and Base Change, Lecture Notes in Mathematics 2156, Springer, 2016.

\bibitem{Kennelly}
A. E. Kennelly,
Equivalence of triangles and three-pointed stars in conducting networks,
Electrical World and Engineer 34 (1899), 413--414.

\bibitem{Kir47}
G. Kirchhoff, \"Uber die Aufl\"osung der Gleichungen, auf welche man bei der Untersuchung der linearen Verteilung galvanischer Str\"ome gef\"uhrt wird,
Ann. Phys. Chem. 72 (1847), 497--508.

\bibitem{Lo}
D. Lorenzini, Arithmetical graphs, Math. Ann. 285 (1989), 481--501.

\bibitem{LoComp}
D. Lorenzini, Groups of components of N\'eron models of Jacobians,
Compositio Math. 73 (1990), 145--160.

\bibitem{Lo08}
D. Lorenzini, Smith normal form and Laplacians, J. Comb. Theory, Ser. B, 98(6) (2008), 1271--1300.

\bibitem{MZ}
G. Mikhalkin, V. Zharkov,
Tropical curves, their Jacobians and theta functions,
in: Curves and Abelian Varieties,
Contemp. Math. 465, Amer. Math. Soc., Providence, RI, 2008, 203--230.

\bibitem{NSW}
J. Neukirch, A. Schmidt, K. Wingberg,
Cohomology of Number Fields, 2nd ed.,
Grundlehren der Mathematischen Wissenschaften 323,
Springer-Verlag, Berlin, 2008.

\bibitem{PS}
B. Poonen, M. Stoll,
The Cassels-Tate pairing on polarized abelian varieties,
Ann. of Math. (2) 150 (1999), 1109--1149.

\bibitem{Ray}
M. Raynaud, Sp\'ecialisation du foncteur de Picard,
Publ. Math. IHES 38 (1970), 27--76.

\bibitem{Padma}
P. Srinivasan, Invariants linked to models of curves over discrete valuation rings,
Ph.D. thesis, MIT, 2016.

\bibitem{Trees}
J.-P. Serre, Trees, translated from the French by J. Stillwell,
Springer-Verlag, Berlin--New York, 1980.

\bibitem{TatC}
J. Tate, On the conjectures of Birch and Swinnerton-Dyer and a
geometric analog, S\'eminaire Bourbaki, 18e ann\'ee, 1965/66, no. 306.

\bibitem{TatA}
J. Tate, Algorithm for determining the type of a singular fiber in
an elliptic pencil, in: Modular Functions of One Variable IV,
Lect. Notes in Math. 476, B. J. Birch and W. Kuyk, eds., Springer-Verlag,
Berlin, 1975, 33--52.

\end{thebibliography}
\end{document}